\documentclass[11.5pt, oneside]{amsart}
\usepackage{amsmath, amscd, amssymb}
\usepackage[frame,cmtip,arrow,matrix,line,graph,curve]{xy}
\usepackage{color}
\usepackage[pdftex,colorlinks,backref=page,citecolor=blue]{hyperref}

\numberwithin{equation}{section}

\newcommand{\comments}[1]{}

\newcommand{\CC}{\mathbb{C}}

\newcommand{\PP}{\mathbb{P}}
\newcommand{\QQ}{\mathbb{Q}}

\newcommand{\ZZ}{\mathbb{Z}}

\newcommand{\bB}{\mathbf{B}}
\newcommand{\bC}{\mathbf{C}}
\newcommand{\bD}{\mathbf{D}}
\newcommand{\bH}{\mathbf{H}}
\newcommand{\bK}{\mathbf{K}}
\newcommand{\bM}{\mathbf{M}}
\newcommand{\bN}{\mathbf{N}}

\newcommand{\bR}{\mathbf{R}}

\newcommand{\bX}{\mathbf{X}}
\newcommand{\bY}{\mathbf{Y}}
\newcommand{\bZ}{\mathbf{Z}}

\newcommand{\cal}{\mathcal}

\def\cC{{\cal C}}

\def\cL{{\cal L}}
\def\cM{{\cal M}}

\def\cO{{\cal O}}

\def\cQ{{\cal Q}}

\newcommand{\ses}[3]{0\rightarrow{#1}\rightarrow{#2}\rightarrow{#3}\rightarrow 0}

\def\and{\quad{\rm and}\quad}

\def\and{\quad\text{and}\quad}

\def\PP{\mathbb{P}}
\def\CC{\mathbb{C}}

\def\cO{\mathcal{O}}

\def\git{/\!/ }

\def\lr{\rightarrow}

\def\PGL{\mathrm{PGL}}
\def\SL{\mathrm{SL}}
\def\gl{\mathfrak{gl}}

\DeclareMathOperator{\Ext}{Ext} 
\DeclareMathOperator{\cExt}{\mathcal{E}xt}
\DeclareMathOperator{\Hom}{Hom} 
\DeclareMathOperator{\cHom}{\mathcal{H}om}

 \DeclareMathOperator{\Stab}{Stab}
\DeclareMathOperator{\rank}{rank}

\newtheorem{prop}{Proposition}[section]
\newtheorem{theo}[prop]{Theorem}
\newtheorem{lemm}[prop]{Lemma}
\newtheorem{coro}[prop]{Corollary}

\theoremstyle{definition}

\newtheorem{defi}[prop]{Definition}
\newtheorem{rema}[prop]{Remark}
\newtheorem{ques}[prop]{Question}

\title[Moduli of sheaves, FM transforms, and partial desingularization]{Moduli of sheaves, Fourier-Mukai transform, and partial desingularization}

\author{Kiryong Chung}
\address{Department of Mathematics Education, Kyungpook National University, 80 Daehakro, Bukgu, Daegu 41566, Korea}
\email{krchung@knu.ac.kr}

\author{Han-Bom Moon}
\address{Department of Mathematics, Fordham University, Bronx, NY 10458, USA}
\email{hmoon8@fordham.edu}


\begin{document}
\begin{abstract}
We study birational maps among 1) the moduli space of semistable sheaves of Hilbert polynomial $4m+2$ on a smooth quadric surface, 2) the moduli space of semistable sheaves of Hilbert polynomial $m^{2}+3m+2$ on $\mathbb{P}^{3}$, 3) Kontsevich's moduli space of genus-zero stable maps of degree 2 to the Grassmannian $Gr(2, 4)$. A regular birational morphism from 1) to 2) is described in terms of Fourier-Mukai transforms. The map from 3) to 2) is Kirwan's partial desingularization. We also investigate several geometric properties of 1) by using the variation of moduli spaces of stable pairs.
\end{abstract}
\maketitle


\section{Introduction}\label{sec:intro}

In this paper, $V$ is a fixed complex vector space of dimension 4. Let $\{x, y, z, w\}$ be a basis of $V^{*}$. In $\PP^{3} = \PP(V)$, $Q = Z(xy-zw)$ is a smooth quadric surface. Let $Gr(2, 4)$ be the space of lines in $\PP(V)$.

The aim of this paper is to understand birational maps between several moduli spaces. Here are the three main characters in this article.

\begin{itemize}
\item The space $\bM_{2}$ is the moduli space of semistable sheaves $F$ on $Q$ with Hilbert polynomial $4m+2$ and $c_{1}(\mathrm{Supp(F)}) = (2,2)$.
\item The space $\bR$ is the moduli space of semistable sheaves on $\PP^{3}$ with Hilbert polynomial $m^{2}+3m+2$.
\item The space $\bK$ is Kontsevich's moduli space of genus-zero stable maps of degree 2 to $Gr(2, 4)$.
\end{itemize}

Note that all of these moduli spaces are main ingredients in curve counting theories, in particular, generalized Donaldson-Thomas (or Pandharipande-Thomas) theory and Gromov-Witten theory.
One of the central problems in modern algebraic geometry is to understand the relation between these curve counting theories. For an excellent survey, see \cite{PT14}.

By using the standard deformation theory on each moduli space, it is not difficult to show that all of those three moduli spaces are 9-dimensional varieties. Although it does not quite seem obvious, they are indeed \emph{birational}.

The fact that they share a common open dense subset can be seen as follows. (We provide details in Section \ref{sec:rationality}.) A general point on $\bM_{2}$ parametrizes a line bundle $\cO_{C}(p+q)$ of degree 2 on a smooth elliptic quartic curve $C$ on $Q$. On $\bR$, a general point parametrizes a pair $(Q', \ell)$ where $Q'$ is a smooth quadric surface in $\PP^{3}$ and $\ell$ is one of two rulings on $Q' \cong \PP^{1} \times \PP^{1}$. (This pair was called a `regulus' in \cite{NT90}.) For a general point $[\cO_{C}(p+q)] \in \bM_{2}$, $C$ is the intersection of two smooth quadrics $Q$ and $Q_{1}$. By imposing the condition that the line $\langle p, q\rangle$ is in $Q_{1}$, we can uniquely determine a quadric and assign a ruling class $\langle p, q\rangle$. Thus we can assign $[(Q_{1}, \langle p, q\rangle)]$ on $\bR$. One can check that this is indeed a bijection on an open dense subset. Furthermore, from $(Q', \ell)$, we can obtain a family over $\PP^{1}$ of lines in $Q' \subset \PP^{3}$ which is in the class of $\ell$. So we have a map $\PP^{1} \to Gr(2, 4)$. This map has degree two, hence it is an element of $\bK$.

\subsection{Main results}

In this paper, we study the birational maps between these spaces. More precisely, we show that the birational morphisms can be described in terms of Fourier-Mukai (shortly, FM) transforms and Kirwan's partial desingularization.

\begin{theo}[\protect{Theorem \ref{mainthm1}}]\label{mainthm1intro}
There exists a birational morphism
\[
	\Psi: \bM_2 \longrightarrow \bR.
\]
The map $\Psi$ is a composition of two FM transforms $\Psi_{1} : D^{b}(Q) \to D^{b}((\PP^{3})^{*})$ and $\Psi_{2} : D^{b}((\PP^{3})^{*}) \to D^{b}(\PP^{3})$. Furthermore, $\Psi$ is a smooth blow up of two points on $\bR$.
\end{theo}

By the work of Le Potier in \cite{LP93a}, $\bR$ can be constructed as an elementary ($\PGL_{2} \times \PGL_{2}$)-GIT quotient of a projective space $\PP(\CC^{2}\otimes V, \CC^{2})$. This GIT quotient has strictly semistable points, and $\bR$ has some bad singularities. Surprisingly, its partial resolution of singularities has a very different moduli theoretic interpretation.

\begin{theo}[\protect{Theorem \ref{thm:pardesing}}]\label{thm:pardesingintro}
There is a birational morphism
\[
	\overline{\pi} : \bK \to \bR,
\]
which is Kirwan's partial desingularization of $\bR$. It is a composition of two blow ups along the singular loci of $\bR$.
\end{theo}

In summary, we have a common contraction
\[
	\xymatrix{\bM_{2} \ar[rd]_{\Psi} && \bK \ar[ld]^{\overline{\pi}}\\
	& \bR.}
\]
Because these two maps are blow ups along disjoint centers, on the log minimal model program (MMP) of $\bM_{2}$ or $\bK$, the other space does not appear.

\medskip

In the past few years, there has been an explosion of the study of birational geometry of moduli spaces of semistable sheaves from the viewpoint of log MMP. Indeed many of these birational models are moduli spaces of Bridgeland stable objects in the derived category of coherent sheaves (\cite{ABCH13, BM13, BM14, BC13, BMW14, CC13b, CHW14}).

Our work arose from the study of the geometry of $\bM_{2}$ in the viewpoint of the Bridgeland wall-crossing on moduli spaces of torsion sheaves on rational surfaces with Picard number $\ge 2$. During this study, the first author observed that by applying the work of Le Potier in \cite[Section 4]{LP93a}, one may describe the birational morphism between its birational models in terms of FM transforms. In \cite{LP93a}, Le Potier applied a FM transform to obtain a birational map from a moduli space of torsion sheaves on $\PP^{2}$ with Euler characteristic zero to a moduli space of torsion-free sheaves on its dual plane $(\PP^{2})^{*}$. Theorem \ref{mainthm1intro} tells us that a FM transform may provide a birational map between two moduli spaces of sheaves on \emph{different} projective varieties.

\subsection{Desingularization of $\bM_2$}
Along the locus of strictly semistable sheaves, $\bM_2$ has some singularities. Finding a desingularization of a singular moduli space is a problem with a very long history and interesting implications. Finding a resolution of a given moduli space is useful to study the geometry of the singular moduli space. For instance in \cite{Kir86b}, by using her desingularization method, Kirwan computed the rank of intersection cohomology groups of the moduli space of vector bundles on a curve when it is singular. In the case of moduli spaces of vector bundles on a curve, several desingularizations with certain moduli theoretic interpretations are constructed in \cite{NR78, Ses77}. Furthermore, the desingularization can be a highly nontrivial example of a variety with a desired geometric property. In \cite{OG99}, O'Grady constructed a new compact hyperk\"ahler manifold by taking a desingularization of a certain singular moduli space of torsion-free sheaves on a K3 surface.

Two main theorems in this paper combined with \cite[Theorem 1.2]{CHK12} provide an explicit desingularization of $\bM_2$. $\bK$ is singular along the locus $\bD$ of stable maps that are two-to-one maps to their images. By \cite[Theorem 1.2]{CHK12}, the blow up of $\bK$ along $\bD$  is a smooth projective variety $\mathbf{CC}_2 := \mathbf{CC}_{2}(Gr(2,4))$, namely, the space of complete conics. On $\bR$, two blow up centers for $\bM_{2} \to \bR$ and $\bK \to \bR$ are disjoint. Thus the fiber product $\widetilde{\bM}_{2} := \bM_{2} \times_{\bR}\mathbf{CC}_{2}$ is a desingularization of $\bM_{2}$. It would be very interesting to find a moduli theoretic interpretation of $\widetilde{\bM}_{2}$.
\[
	\xymatrix{\widetilde{\bM}_{2} \ar[dd] \ar[r] &
	\mathbf{CC}_{2}\ar[d]\\
	& \bK \ar[d] \\ \bM_{2} \ar[r] & \bR}
\]

\begin{ques}
What is the moduli theoretic meaning of $\widetilde{\bM}_{2}$?
\end{ques}

\subsection{Rationality}

The original motivation of this paper was to show the rationality of $\bM_{2}$, which was asked in \cite[Conjecture 35]{BH14}. The fact that $\bM_{2}$ is birational to $\bK$ gives the rationality of $\bM_{2}$, because the rationality of $\bK$, or more precisely, the space of conics in $Gr(2, 4)$ is well-known, for example in \cite[Section 3.10]{NT90} or in \cite{CC10}. On the diagram below, $\bH = \mathrm{Hilb}^{2m+1}(Gr(2, 4))$ is the Hilbert scheme of conics in $Gr(2, 4)$, and $\bC = \mathrm{Chow}_{1,2}(Gr(2, 4))$ is the normalization of the irreducible component of the Chow variety containing smooth conics. All maps are birational morphisms.
\[
	\xymatrix{\bK \ar[rd] \ar@{<-->}[rr] && \bH \ar[ld] \ar[rd]\\
	& \bC && Gr(3, 6)}
\]

\subsection{Geometry of $\bM_{2}$ and its numerical invariants}

Besides the result described above, by using the moduli space $\bM_{2}^{\alpha}$ of $\alpha$-stable pairs, we computed some topological invariants of $\bM_{2}$. In Section \ref{sec:geomofM2}, we study 1) the variation of the moduli space of $\alpha$-stable pairs from $\alpha = \infty$ to $\alpha = 0^{+}$ and 2) the fiber of the map $\phi : \bM_{2}^{0^{+}} \longrightarrow \bM_{2}$. Indeed, when $\alpha = 2$, there is a single wall-crossing which is a composition of a smooth blow up and a smooth blow down. In Section \ref{sec:geomofM2} we show that there is a diagram
\[
	\xymatrix{\bM_{2}^{\infty} \ar[d]_{\pi} \ar[rd] \ar@{<-->}[rr]
	&& \bM_{2}^{0^{+}} \ar[ld] \ar[d]^{\phi}\\
	|\cO_{Q}(2, 2)| & \bM_{2}^{2} & \bM_{2}}
\]
where $\phi$, $\pi$ are algebraic fiber spaces and the other two maps are birational morphisms. As consequences, we compute 1) the singular locus, 2) the rank of the Picard group, and 3) the virtual Poincar\'e polynomial of $\bM_{2}$. Also we obtain a classification of all free resolution types of objects in $\bM_{2}$, which is crucial in Section \ref{sec:M2andR}.

\subsection{Organization of the paper}

This paper is organized in the following way. In Section \ref{sec:preliminaries}, we give definitions of moduli spaces appearing in this paper and review some well-known properties. In Section \ref{sec:geomofM2}, we study the geometry of $\bM_2$ and compute its numerical invariants. Section \ref{sec:rationality} offers an elementary and classical argument to show the rationality of $\bM_{2}$. In the next two sections, we describe the birational map in Section \ref{sec:rationality} in terms of FM transforms and partial desingularization. In Section \ref{sec:M2andR}, we prove Theorem \ref{mainthm1intro}. Finally, in Section \ref{sec:RandK}, we show Theorem \ref{thm:pardesingintro}.

\subsection{Notation and conventions}

We work over the complex number $\CC$. We will denote the direct sum $F^{m}$ of $m$-copies of a coherent sheaf $F$ by $mF$ if there is no confusion. The boldface letters $\bR$, $\bM$, $\bK$, $\bN$ refer to (coarse) moduli or parameter spaces.

\subsection*{Acknowledgements}
Originally, this paper was motivated from our effort to prove the rationality of $\bM_2$, which was conjectured in \cite{BH14}. It is our pleasure to thank Sukmoon Huh for the suggestion of this problem and useful discussions. The authors would like to thank Young-Hoon Kiem for thorough and helpful comments on earlier drafts of this paper. Part of this work has been done while the first named author was working in Korea Institute for Advanced Study as a research fellow. KC was supported by Korea NRF grant 2013R1A1A2006037.


\section{Preliminaries}\label{sec:preliminaries}

In this section, we review definitions and several well-known properties of moduli spaces we will discuss.

\subsection{Moduli space of semistable sheaves}\label{ssec:Simpson}

Let $X$ be a smooth projective variety with fixed polarization $L$. For a coherent sheaf $F$ on $X$, the Hilbert polynomial $P(F)(m)$ is $\chi(F\otimes L^{m})$, where $\chi$ is the Euler characteristic. If the support of $F$ has dimension $d$, $P(F)(m)$ has degree $d$ and it can be written as
\[
	P(F)(m) = \sum_{i=0}^{d}a_{i}\frac{m^{i}}{i!}.
\]
The coefficient $r(F) := a_{d}$ of the highest power is called the multiplicity. The \emph{reduced Hilbert polynomial} is $p(F)(m) := P(F)(m)/r(F)$. A sheaf $F$ is \emph{semistable} if
\begin{itemize}
\item $F$ is pure;
\item for every nonzero proper subsheaf $F' \subset F$, $p(F')(m) \le p(F)(m)$ for $m \gg 0$.
\end{itemize}
We say $F$ is \emph{stable} if the inequality is strict. For each semistable sheaf $F$, there is a filtration (the so-called Jordan-H\"older filtration) $0 = F_{0} \subset F_{1} \subset \cdots \subset F_{n} = F$ such that $\mathrm{gr}_{i}(F) := F_{i}/F_{i-1}$ is stable and $p(F)(m) = p(\mathrm{gr}_{i}(F))(m)$ for all $i$. Finally, two semistable sheaves $F_{1}$ and $F_{2}$ are \emph{$S$-equivalent} if $\mathrm{gr}(F_{1}) \cong \mathrm{gr}(F_{2})$ where $\mathrm{gr}(F) := \oplus_{i}\mathrm{gr}_{i}(F)$.

In his monumental paper \cite{Sim94}, Simpson proved that there is a projective coarse moduli space $\bM_{L}(X, P(m))$ of $S$-equivalence classes of semistable sheaves of fixed Hilbert polynomial $P(m)$. There is an open subset $\bM_{L}(X, P(m))^{s}$ parametrizing stable sheaves. For a great introduction and the details of proofs and related results, see \cite{HL10}.

In this paper, we will focus on two examples. Let $Q \subset \PP(V)$ be a smooth quadric surface. Since $Q \cong \PP^{1}\times \PP^{1}$, we may denote a line bundle as $\cO_{Q}(a, b)$ for two integers $a, b$. Take $L = \cO_{Q}(1,1)$ as a polarization. Consider $\bM_{L}(Q, 4m+2)$. This moduli space has several connected components, which are parametrized by $\beta \in H_{2}(Q, \ZZ)$ representing the scheme theoretic support of $F$. Thus we may write
\[
	\bM_{L}(Q, 4m+2) = \bigsqcup_{\beta \in H_{2}(Q, \ZZ)}
	\bM_{L}(Q, \beta, 4m+2).
\]

\begin{defi}\label{def:M2}
Let $\bM_{2} := \bM_{L}(Q, c_{1}(\cO_{Q}(2, 2)), 4m+2)$. So $\bM_{2}$ is the moduli space of $S$-equivalence classes of semistable torsion sheaves of multiplicity 4, of the support class $c_{1}(\cO_{Q}(2,2))$ and $\chi = 2$. It is a $c_1^2+1=9$-dimensional variety (see \cite[Proposition 2.3]{LP93a}).
\end{defi}

\begin{defi}\label{def:R}
Let $\bR := \bM_{\cO_{\PP^{3}}(1)}(\PP^{3}, m^{2}+3m+2)$. It is also a 9-dimensional variety.
\end{defi}

\subsection{Moduli space of stable pairs}\label{ssec:modofpair}

To analyze some geometric properties of $\bM_{2}$, we will use the moduli space of stable pairs. As in the previous section, let $Q \cong \PP^{1} \times \PP^{1}$ be a smooth quadric surface and let $L = \cO_{Q}(1, 1)$. A pair $(s, F)$ consists of a coherent sheaf $F$ on $Q$ and a nonzero section $\cO_{Q} \stackrel{s}{\to} F$. Fix a positive rational number $\alpha$. A pair $(s,F)$ is called \emph{$\alpha$-semistable} if $F$ is pure and for any subsheaf $F'\subset  F$, the inequality
\[
	\frac{P(F')(m)+\delta\cdot\alpha}{r(F')} \le
	\frac{P(F)(m)+\alpha}{r(F)}
\]
holds for $m\gg 0$. Here $\delta=1$ if the section $s$ factors through $F'$ and $\delta=0$ otherwise. When the strict inequality holds, $(s,F)$ is called as an \emph{$\alpha$-stable} pair. As in the case of sheaves, we can define the Jordan-H\"older filtration and $S$-equivalence classes of pairs.

There exists a projective scheme $\bM_{L}^{\alpha}(Q,P(m))$ parameterizing $S$-equivalence classes of $\alpha$-semistable pairs with Hilbert polynomial $P(m)$ (\cite[Theorem 4.12]{LP93b} and \cite[Theorem 2.6]{He98}). Also, we have a decomposition of the moduli space
\[
	\bM_{L}^{\alpha}(Q,4m+2)=\bigsqcup_{\beta\in H_2(Q, \ZZ)}
	\bM_{L}^{\alpha}(Q,\beta, 4m+2).
\]
We denote the moduli space $\bM_{L}^{\alpha}(Q, c_{1}(\cO_{Q}(2,2)), 4m+2)$ by $\bM_{2}^{\alpha}$. The extremal case that $\alpha$ is sufficiently large (resp. small) is denoted by $\alpha=\infty$ (resp. $\alpha = 0^{+}$).
The deformation theory of pairs has been studied by many authors. For our purpose, see \cite[Corollary 1.6 and Corollary 3.6]{He98}.

\subsection{Kontsevich's moduli space of stable maps}\label{ssec:Kontsevich}

For a smooth projective variety $X$, fix a Chow class $\beta \in A_{1}(X, \ZZ)$. Let $(C, x_{1}, \cdots, x_{n})$ be a projective connected reduced curve with $n$ marked points. A map $f: (C, x_{1}, \cdots, x_{n}) \to X$ is called \emph{stable} if
\begin{itemize}
\item $C$ has at worst nodal singularities;
\item $x_{1}, \cdots, x_{n}$ are $n$ distinct smooth points;
\item $|\mathrm{Aut}(f)| < \infty$, or equivalently, $\omega_{C}(\sum x_{i})$ is $f$-ample.
\end{itemize}
Let $\cM_{g, n}(X, \beta)$ be the moduli stack of $n$-pointed stable maps with $g(C) = g$, $f_{*}[C] = \beta$. It is well-known that $\cM_{g, n}(X, \beta)$ is a proper Deligne-Mumford stack and its coarse moduli space $\bM_{g, n}(X, \beta)$ is a projective scheme (\cite[Theorem 1]{FP97}).

In this paper, let $\bK := \bM_{0, 0}(Gr(2, 4), 2L)$ and let $\bK_{1} := \bM_{0,1}(Gr(2, 4), 2L)$ where $L$ be the effective generator of $A_{1}(Gr(2, 4), \ZZ)$. By \cite[Corollary 1]{KP01}, these are normal irreducible varieties. There is a natural forgetful map $\pi : \bK_{1} \to \bK$ which forgets the marked point. One can compute the dimension of $\bK$ (resp. $\bK_{1}$), which is 9 (resp. 10), by using \cite[Theorem 2]{FP97}.


\section{Geometry of $\bM_{2}$}\label{sec:geomofM2}

In this section, we compute some numerical invariants of $\bM_{2}$ by using the variation of moduli spaces of stable pairs.

\subsection{Wall-crossing and geometry of $\bM_{2}$}

Recall that $\bM_{2}^{\alpha}$ is the moduli space of $\alpha$-semistable pairs on $Q$ with Hilbert polynomial $4m+2$ and support class $c_{1}(\cO_{Q}(2,2))$.

In two extremal cases $\alpha = \infty$ and $\alpha = 0^{+}$, we have well-known structure morphisms on $\bM_{2}^{\alpha}$.

\begin{lemm}\label{lemm:extremalstab}
\begin{enumerate}
\item The moduli space $\bM_2^{\infty}$ is isomorphic to the relative Hilbert scheme
\[
	\mathrm{Hilb}^2(\cC/|\cO_{Q}(2,2)|)
\]
of two points over the universal quartic curve $\cC\lr |\cO_{Q}(2,2)|$ in $Q$. The latter space is a $\PP^6$-bundle over the Hilbert scheme $\mathrm{Hilb}^{2}(Q)$ of two points.
\item There exists a forgetful map $\phi:\bM_2^{0^{+}}\longrightarrow \bM_2$ which associates $F$ to the pair $(s,F)$.
\end{enumerate}
\end{lemm}

\begin{proof}
The isomorphism in Item (1) is given in the following way (see \cite[Proposition B.8]{PT10} and \cite[\S4.4]{He98}). For a pair $(Z,C) \in \mathrm{Hilb}^2(\cC/|\cO_{Q}(2,2)|)$, we associate an $\infty$-stable pair $(s,I_{Z,C}^D:=\cExt_Q^1(I_{Z,C},\omega_Q))\in \bM_2^{\infty}$ where $s$ is constructed in the following way. Apply the functor $\cHom_Q(-,\omega_{Q})$ in the short exact sequence $\ses{I_{Z,C}}{\cO_C}{\cO_Z}$. Then we obtain a short exact sequence
\begin{equation}\label{dualmap}
	\ses{\cExt_Q^1(\cO_C, \omega_Q)}{I_{Z,C}^D}{\cExt_Q^2(\cO_Z,
	\omega_{Q})}.
\end{equation}
The first map is injective since $\cO_Z$ is supported on a zero-dimensional scheme (\cite[Proposition 1.1.6]{HL10}) and the second map is surjective because of the pureness of $\cO_C$. One can easily see $\cExt_Q^1(\cO_C, \omega_Q)\cong\cO_C$  by using a resolution of $\cO_C$. By composing the restriction map $\cO_Q\lr\cO_C$ and the second map in \eqref{dualmap}, we have a canonical section $s:\cO_Q\lr I_{Z,C}^D$. Note that the length $l(\cExt_Q^2(\cO_Z, \omega_{Q}))=2$, so $I_{Z, C}^{D}$ has Hilbert polynomial $4m+2$.

The second part of Item (1) comes from the fact that $h^0(I_{\{p,q\}}\otimes \cO_Q(2,2))=7$ for any two points $\{p,q\}$ on $Q$ (possibly $p=q$) (see \cite[Example 6.1]{BC13}).

Now choose $\alpha<\frac{1}{4}$. Then one can easily check that if $(s, F)$ is an $\alpha$-stable pair, then $F$ is semistable. Hence the forgetful functor $(s, F) \mapsto F$ indeed defines a map $\phi : \bM_{2}^{0^{+}} \to \bM_{2}$. Thus we have (2).
\end{proof}

Now we study the wall-crossing when the stability parameter $\alpha$ decreases. We denote the pair $(s,F)$ by $(1,F)$ if the section $s\in H^0(F)$ is non-zero and $(0,F)$ otherwise.

\begin{prop}\label{wall}
There is a simple wall crossing
\[
	\xymatrix{\bM_2^{\infty}\ar[dr]\ar@{<--}[r]&&
	\bM_2^{0^{+}}\ar[dl]\ar@{<--}[l]\\
	&\bM_2^{2}&}
\]
at $\alpha=2$. When $\alpha = 2$, an irreducible component of the flipping locus is the subvariety parametrizing $S$-equivalence classes of $(1,F)=(1,\cO_C)\oplus (0,\cO_L)$ for a rational cubic curve $C$ of the class $\cO_Q(2, 1)$ (resp. $\cO_Q(1,2)$) and a line $L$ of the class $\cO_Q(0, 1)$ (resp. $\cO(1,0)$). Furthermore, the moduli space $\bM_2^{0^{+}}$ is a smooth variety of dimension $10$.
\end{prop}

Let $(1,F)\in \PP\Ext_Q^1((1,\cO_C),(0,\cO_L))\cong \PP^2$ be an $\infty$-stable pair. Then it fits into the exact sequence $\ses{(0,\cO_L)}{(1,F)}{(1,\cO_C)}$. Geometrically, the pair $(1,F)$ corresponds to a pair of subschemes $(Z, L\cup C)$ such that $Z$ is contained in $L$. After the wall-crossing, the pair $(1, F)$ is replaced by a pair $(1, G)$ in an extension obtained by switching the sub pairs and the quotient pairs: $\ses{(1,\cO_C)}{(1,G)}{(0,\cO_L)}$ (\cite[Proof of Lemma 4.24]{He98}). Hence we obtain $0^+$-stable pairs $(1,G)\in \PP\Ext_Q^1((0,\cO_L),(1,\cO_C))\cong \PP^1$.

\begin{proof}[Proof of Proposition \ref{wall}]
Consider a pair $(s,F)$ on $Q$. From an upper bound of the dimension of the space of global sections of semistable sheaves (\cite[Theorem 1.3]{CC13a}), one can see that the possible wall is one of the following:
\begin{enumerate}
\item $(1,F_{4m+2})=(1,F_{3m})\oplus (0,F_{m+2})$ or
\item $(1,F_{4m+2})=(1,F_{3m+1})\oplus (0,F_{m+1})$.
\end{enumerate}
Here $F_k$ means a semistable sheaf with Hilbert polynomial $k$.

But the case (1) cannot occur because $F_{3m}\cong \cO_C$ for some plane cubic curve $C$ and thus it is not contained in the quadric surface $Q$.
Hence the wall occurs if and only if $(1,F_{4m+2})=(1,F_{3m+1})\oplus (0,F_{m+1})$. By the classification of stable sheaves with small Hilbert polynomials in $\PP^3$ (\cite{FT04}), we know that $F_{3m+1}\cong \cO_C$ and $F_{m+1}=\cO_L$ for some twisted cubic curve $C$ and a line $L$.

We show that $\bM_2^{0^{+}}$ is a smooth variety. Since the complement of the wall crossing locus in $\bM_2^{0^{+}}$ is isomorphic to the corresponding locus in $\bM_2^{\infty}$ and $\chi(\Ext^{\bullet}((1,F), (1,F)))=10$, it is sufficient to check that the obstruction space $\Ext^2((1,F),(1,F))$ vanishes for the pairs $(1,F)$ in the wall crossing locus of $\bM_2^{0^{+}}$ (\cite[Corollary 3.10, Theorem 3.12]{He98}). The pair $(1,F)$ fits into a nontrivial extension $\ses{(1,\cO_C)}{(1,F)}{(0,\cO_L)}$. This implies that $H^1(F)=0$ and $F$ is a semistable sheaf on $Q$. From \cite[Corollary 1.6]{He98}, we have an exact sequence
\[
	\Hom(\CC, H^1(F))\lr \Ext^2((1,F),(1,F)) \lr \Ext^2(F,F)\cong
	\Ext^0(F,F(-2,-2))^*
\]
where the last isomorphism is given by the Serre duality. But  $\Ext^0(F,F(-2,-2))=0$ because of the semistability of $F$.
\end{proof}

\begin{lemm}\label{lemm:basedecom}
Let $F \in \bM_{2}\setminus \bM_{2}^{s}$ be a polystable sheaf, that is, a semistable sheaf satisfying $F \cong \mathrm{gr}(F)$. Then $F \cong \cO_{C_{1}}\oplus \cO_{C_{2}}$ for two conics $C_1$ and $C_2$ on $Q$.
\end{lemm}

\begin{proof}
Note that if $F \in \bM_2$ is strictly semistable, then the destabilizing subsheaf must have the Hilbert polynomial $2m+1$, because this is the only nontrivial integer valued polynomial dividing $4m+2$. The only semistable sheaf with this Hilbert polynomial is $\cO_{C_{1}}$ for some conic $C_{1}$ in $Q$. The quotient sheaf also has the same Hilbert polynomial $2m+1$ and hence it is $\cO_{C_{2}}$ for some conic $C_{2}$ in $Q$. Therefore $F \cong \mathrm{gr}(F) \cong \cO_{C_{1}} \oplus \cO_{C_{2}}$.
\end{proof}

\begin{coro}[\protect{\cite[Corollary 23]{BH14}}]\label{cor:sslocus}
The strictly semistable locus $\bM_{2}\setminus \bM_{2}^{s}$ is isomorphic to $\PP^{3}\times \PP^{3}/\ZZ_{2}$.
\end{coro}

In the next proposition we compute analytic locally the fibers of the last forgetful map $\phi : \bM_{2}^{0^{+}} \longrightarrow \bM_{2}$.

\begin{prop}\label{prop:fiberofphi}	
Let $\phi : \bM_{2}^{0^{+}} \to \bM_{2}$ be the forgetful map $(s, F) \mapsto F$.
\begin{enumerate}
\item Over $\bM_{2}^{s}$, $\phi$ is a $\PP^1$-fibration;
\item Let $D \subset \PP^{3} \times \PP^{3}/\ZZ_{2}$ be the diagonal. Over $(\PP^{3} \times \PP^{3}/\ZZ_{2} \setminus D) \subset \bM_{2}\setminus \bM_{2}^{s}$, $\phi$ is a $\PP^2\cup_{p=q}\PP^{2}$-fibration over $\PP^3\times \PP^3/\ZZ_{2} \setminus D$, where $\PP^{2}\cup_{p=q}\PP^{2}$ is a reducible variety obtained by gluing two $\PP^{2}$'s along two closed points $p$ and $q$. On $(\PP^{2}\setminus \{p\}) \sqcup (\PP^{2}\setminus\{q\})$, the fiber parametrizes extensions $(1,F)\in\PP\Ext^1((1,\cO_{C_1}),(0,\cO_{C_2}))\cong \PP^{2}$ for $C_1\neq C_2 \in |\cO_{Q}(1,1)| = \PP^{3}$. The gluing point $p$ ($=q$) parametrizes the unique pair $(s,\cO_{C_1}\oplus \cO_{C_2})$ for conics $C_1 \neq C_2$.
\item Over $D \cong \PP^{3} = |\cO_{Q}(1,1)| \subset \bM_{2} \setminus \bM_{2}^{s}$, $\phi$ is a $\PP^2$-fibration parametrizing pairs $(1,F)\in\PP\Ext^1((1,\cO_{C}),(0,\cO_{C}))\cong \PP^{2}$ for $C \in |\cO_{Q}(1,1)|$.
\end{enumerate}
\end{prop}

\begin{proof}
Let $F$ be a stable sheaf. The inverse image $\phi^{-1}([F])$ consists of the $0^+$-stable pairs $(s,F)$ such that $s\subset H^0(F)$. This implies that the fiber is isomorphic to $\PP^1$ because $h^0(F)=2$. This proves Item (1).

Suppose that $(s,F)$ is in the inverse image of the strictly semistable locus. We may assume that $\mathrm{gr}(F)=[\cO_{C_1}\oplus \cO_{C_2}]$ for two conics $C_1, C_2 \in |\cO_{Q}(1,1)|$. First assume that $C_1\neq C_2$ and $F\ncong \cO_{C_1}\oplus \cO_{C_2}$. Then by \cite[Corollary 1.6]{He98}, there exists a long exact sequence
\[
\Ext^0(\cO_{C_1},\cO_{C_2} )\lr \Hom (\CC,H^0(\cO_{C_2}))\lr \Ext^1((1,\cO_{C_1}),(0,\cO_{C_2}))
\]
\[
\lr \Ext^1(\cO_{C_1},\cO_{C_2} )\lr \Hom (\CC,H^1(\cO_{C_2})).
\]
Since $C_1\neq C_2$, we have $\Ext^0(\cO_{C_1},\cO_{C_2} )=0$. Also, $\Hom (\CC,H^0(\cO_{C_2}))\cong \CC$ and $\Hom (\CC,H^1(\cO_{C_2}))=0$. By using the standard resolution of $\cO_{C_1}$, we can compute that $\Ext^1(\cO_{C_1},\cO_{C_2} )\cong \CC^2$. Thus we obtain a short exact sequence
\[
\ses{\Hom (\CC,H^0(\cO_{C_2}))\cong\CC}{\Ext^1((1,\cO_{C_1}),(0,\cO_{C_2})) }{ \Ext^1(\cO_{C_1},\cO_{C_2} )\cong \CC^2}.
\]

The variety $\PP\Ext^1((1,\cO_{C_1}),(0,\cO_{C_2}))\cong \PP^2 \setminus \PP\Hom (\CC,H^0(\cO_{C_2}))\cong\PP^{2} - \mathrm{pt}$ (the point corresponds to the trivial extension) parametrizes stable pairs whose underlying sheaves are nontrivial extensions in the fiber $\phi^{-1}([F])$. When $F \cong \cO_{C_1}\oplus \cO_{C_2}$ for $C_1\neq C_2 \in |\cO_{Q}(1,1)|$, the automorphism group of $F$ is isomorphic to $\CC^*\times \CC^*$, which acts on the section space $H^0(F)\cong H^0(\cO_{C_1})\oplus H^0(\cO_{C_2}) \cong \CC\oplus \CC$ diagonally. Now if $(s,F)=((s_1,s_2),F)$ where $s_1\cdot s_2=0$, then one can easily check that $(s,F)$ is not $0^+$-stable. If $s_1\cdot s_2\neq 0$, then there is a unique  $0^+$-stable pair $(s,F)$ because $\mathrm{Aut}(F)\cong \CC^*\times \CC^*$.

Finally, we prove (3). Let $\mathrm{gr}(F)\cong\cO_{C}\oplus \cO_{C}$ for a conic $C \in |\cO_{Q}(1,1)|$. Then by the same computation as in the case (2), one can see that $\PP\Ext^1((1,\cO_{C}),(0,\cO_{C}))\cong \PP^{2}$.
\end{proof}

\subsection{Some numerical invariants}

Here we leave computations of some numerical invariants of $\bM_{2}$.

\begin{prop}\label{pic}
The rank of the Picard group of $\bM_2$ is 3. Furthermore, $\bM_{2}$ is $\QQ$-factorial.
\end{prop}

\begin{proof}
In Theorem \ref{mainthm1} we will show that $\mathrm{Pic}(\bM_{2})$ is a blow up of $\bR$ at two smooth points. Since $\mathrm{rank}\;\mathrm{Pic}(\bR) \ge 1$, $\mathrm{rank} \;\mathrm{Pic}(\bM_{2}) \ge 3$.

On the other hand, for a smooth projective surface $X$ with $q(X) = 0$, it is well-known that $\mathrm{rank}\;\mathrm{Pic}(\mathrm{Hilb}^{n}(X)) = \mathrm{rank}\;\mathrm{Pic}(X) + 1$ (\cite{Fo73}). So $\mathrm{rank}\;\mathrm{Pic}(\mathrm{Hilb}^2(Q))=3$. Then by Item (1) of Proposition \ref{wall}, $\mathrm{rank}\;\mathrm{Pic}(\bM_{2}^{\infty})=4$. Note that $\bM_{2}^{\infty}$ is isomorphic to $\bM_{2}^{0^{+}}$ up to codimension $2$. So $\mathrm{rank}\;\mathrm{Pic}(\bM_{2}^{0^{+}}) = 4$.

By Proposition \ref{prop:fiberofphi}, the forgetful map $\phi : \bM_{2}^{0^{+}} \to \bM_{2}$ is a dominant morphism between two normal varieties with connected fibers. Thus $\phi$ is an algebraic fiber space. The same is true for its restriction to $\bM_{2}^{0^{+}}|_{\phi^{-1}(\bM_{2}^{s})} \to \bM_{2}^{s}$. Thus $\mathrm{Pic}(\bM_{2}^{s})$ is embedded into $\mathrm{Pic}(\bM_{2}^{0^{+}}|_{\phi^{-1}(\bM_{2}^{s})})$. Moreover, since there is a relative ample divisor on $\bM_{2}^{0^{+}}|_{\phi^{-1}(\bM_{2}^{s})}$ which is not a pull-back of a divisor on $\bM_{2}^{s}$, $\mathrm{rank} \;\mathrm{Pic}(\bM_{2}^{0^{+}}|_{\phi^{-1}(\bM_{2}^{s})}) > \mathrm{rank}\; \mathrm{Pic}(\bM_{2}^{s})$. From the dimension-counting using the explicit description of the base and fibers in Items (2) and (3) of the same proposition, $\mathrm{codim}_{\bM_2}(\bM_2^s)^{c}\geq 2$ and $\mathrm{codim}_{\bM_2^{0^{+}}}\phi^{-1}(\bM_2^s)^c\geq 2$. Thus $\mathrm{rank}\;\mathrm{Pic}(\bM_2^{0^{+}}|_{\phi^{-1}(\bM_2^s)}) = \mathrm{rank}\;\mathrm{Pic}(\bM_{2}^{0^{+}})$ and $\mathrm{rank}\;\mathrm{Cl}(\bM_2^s) =  \mathrm{rank}\;\mathrm{Cl}(\bM_2)$.

In summary,
\[
	3 \le \mathrm{rank}\; \mathrm{Pic}(\bM_{2}) \le
	\mathrm{rank}\; \mathrm{Cl}(\bM_{2}) =
	\mathrm{rank}\; \mathrm{Cl}(\bM_{2}^{s}) =
	\mathrm{rank}\; \mathrm{Pic}(\bM_{2}^{s})
\]
\[
	\le \mathrm{rank}\;
	\mathrm{Pic}(\bM_{2}^{0^{+}}|_{\phi^{-1}(\bM_{2}^{s})}) - 1
	= \mathrm{rank}\; \mathrm{Pic}(\bM_{2}^{0^{+}}) - 1 = 3.
\]
Therefore $\rank\;\mathrm{Pic}(\bM_{2}) = 3$ and $\bM_{2}$ is $\QQ$-factorial.
\end{proof}

For a variety $X$, the \emph{virtual Poincar\'e polynomial} of $X$ is defined by
\[
	P(X)=\sum (-1)^{i+j} \dim_{\QQ} \mathrm{gr}_{W}^{j}
	H_{c}^{i}(X,\QQ) p^{i/2},
\]
where $\mathrm{gr}_{W}^{j}H_{c}^{i}(X, \QQ)$ is the $j$-th weight-graded piece of the mixed Hodge structure on the $i$-th cohomology of $X$ with compact supports. Since odd cohomology groups of moduli spaces of our interest always vanish, their virtual Poincar\'e polynomial is a \emph{polynomial} indeed. For the motivic properties of the virtual Poincar\'e polynomial, see \cite[Section 2]{Mun08}.

The classification of stable pairs in Proposition \ref{prop:fiberofphi} enables us to compute the virtual Poincar\'e polynomial of $\bM_2$.
\begin{coro}\label{virpoinm2}
The virtual Poincar\'e polynomial of $\bM_2$ is given by
\[
	P(\bM_2)=p^9+3p^8+4p^7+3p^6+3p^5+2p^4+3p^3+3p^2+3p+1.
\]
In particular, the virtual Euler number is $e(\bM_2)=26$.
\end{coro}

\begin{proof}
We may assume that all fibration structures in Proposition \ref{prop:fiberofphi} are Zariski locally trivial. Indeed, for the strata $\bM_{2}^{s}$ and $\PP^{3} \subset \bM_{2}$ in (1) and (3) of Proposition \ref{prop:fiberofphi}, they are projective space fibrations. Therefore we can apply \cite[p.21 (6)]{Mun08}. For $\PP^{3} \times \PP^{3}/\ZZ_{2}\setminus D$, if we remove the section $p$, we obtain $(\PP^{2}\setminus\{p\}) \sqcup (\PP^{2}\setminus\{q\})$-fibration. By taking a degree 2 \'etale cover $\PP^{3} \times \PP^{3} \setminus D$ of $\PP^{3} \times \PP^{3}/\ZZ_{2} \setminus D$, we obtain two disjoint $(\PP^{2}\setminus \mathrm{pt})$-fibrations. By adding two sections, we obtain two $\PP^{2}$-fibrations. By \cite[p.21 (6)]{Mun08} again, the computational result is the same as in the case of Zariski bundles.

By Proposition \ref{wall},
\begin{align*}
P(\bM_2^{0^{+}})&=P(\bM_2^{\infty})+2(P(\PP^1)-P(\PP^2))P(\PP^5)P(\PP^1)  \\
&=P(\mathrm{Hilb}^{2}(Q))\cdot P(\PP^6)+2(P(\PP^1)-P(\PP^2))P(\PP^5)P(\PP^1).
\end{align*}
But $P(\mathrm{Hilb}^{2}(Q))=\frac{1}{2}(P(Q)^2-P(Q))+P(Q)\cdot P(\PP^1)$. Hence
\[
	P(\bM_2^{0^{+}})=
	p^{10}+4p^9+8p^8+9p^7+10p^6+10p^5+10p^4+9p^3+8p^2+4p+1.
\]
On the other hand, by the proof of Proposition \ref{prop:fiberofphi}, we have
\[
	P(\bM_2^{0^{+}})=P(\PP^1)P(\bM_2^s)+ (2P(\PP^2)-1)
	(P(\mathrm{Sym}^2\PP^3)\textendash
	P(\PP^3))+P(\PP^2)\cdot P(\PP^3)
\]
and thus obtain $P(\bM_2^s)$. Finally, $P(\bM_2)=P(\bM_2^s)+P(\mathrm{Sym}^2\PP^3)$.
\end{proof}

\subsection{Free resolution of elements of $\bM_{2}$}

Another application of the wall crossing analysis is the classification of free resolutions of elements in $\bM_{2}$. We will use the result of this section in Section \ref{sec:M2andR}.

\begin{prop}\label{resol}
Let $F\in \bM_2$. Then $F$ has one of the following free resolutions:
\begin{enumerate}
\item $\ses{2\cO_{Q}(-1,-1)}{2\cO_{Q}}{F}$;
\item $\ses{\cO_{Q}(-1,-2)}{\cO_{Q}(1,0)}{F}$;
\item $\ses{\cO_{Q}(-2,-1)}{\cO_{Q}(0,1)}{F}$.
\end{enumerate}
\end{prop}

\begin{proof}
We describe all possible resolutions of underlying sheaves, for stable pairs $(s, F) \in \bM_2^{0^{+}}$. If a pair $(s, F)$ does not belong to the wall crossing locus in $\bM_2^{0^{+}}$, then $F\cong I_{\{p,q\},C}^D$ for some $p ,q \in C$ (\cite[Proposition B.8]{PT10} and \cite[\S4.4]{He98}). Here $D$ denotes the dual $G^D:=\cExt_Q^1(G, \omega_{Q})$. If the line $L=\langle p,q\rangle$ generated by $p$ and $q$ is not one of two rulings of $Q$, then $\{p,q\}$ is a complete intersection of two conics in $Q$. Hence we have a short exact sequence $\ses{\cO_{Q}(-2,-2)}{2\cO_{Q}(-1,-1)}{I_{\{p,q\},Q}}$. By performing the mapping cone operation on $\ses{I_{C,Q}\cong \cO_{Q}(-2,-2)}{I_{\{p,q\},Q}}{I_{\{p,q\},C}}$, we obtain a free resolution
\[
	\ses{2\cO_{Q}(-2,-2)}{2\cO_{Q}(-1,-1)}{I_{\{p,q\},C}}.
\]
By taking $\cHom_{Q}(-, \omega_{Q})$, we can obtain the case (1). If the line $L=\langle p,q\rangle$ is a ruling of $Q$, then $I_{\{p,q\},C}\cong \cO_C(-1,0)$ or $\cO_C(0,-1)$. By taking $\cHom_{Q}(-, \omega_{Q})$ to the resolution of $\cO_{C}(-1,0)$ or $\cO_{C}(0,-1)$, we get the cases (2) or (3).

Now assume that the $0^+$-stable pair $(1,F)$ is contained in the wall crossing locus. Then $(1,F)$ fits into the short exact sequence
\[
	\ses{(1,\cO_C)}{(1,F)}{(0,\cO_L)}
\]
for a cubic curve $C$ and a line $L$ in $Q$. If $F$ is strictly semistable, then $\mathrm{gr}(F)=\cO_{C_1}\oplus \cO_{C_2}$ for two conics $C_i$, and by the horseshoe lemma type argument, $F$ has the resolution of the form (1). Now assume that $F$ is stable. Then every pair $(s,F)$ with a nonzero section $s\in H^0(F)=\CC^2$ is $0^+$-stable. Also, the pair $(1,F)$ in our consideration is exactly given by the unique (up to scalar) nonzero section $1 \in H^0(\cO_C)=\CC\subset H^0(F)$. We can construct a flat family of pairs such that the underlying sheaf is constant but the section does vary. Let
\[
	\ses{(0,\cO_C)}{(s_t,F)}{(s_t,\cO_L)}
\]
be the non-split extension such that $\mathrm{lim}_{t\lr 0}(s_t, F)=(1,F)$ for $s_t\in H^0(F)\setminus H^0(\cO_C)$. Then obviously $(s_t,F)$ for all $t\neq 0$ is contained in the complement of the wall crossing locus. Hence, as we discussed before, the possible resolutions of the underlying sheaf $F$ are given by the list. This proves the claim.
\end{proof}

\begin{defi}\label{def:Ddivisor}
Let $D_{1,0}$ (resp. $D_{0,1}$) be the locus of semistable sheaves $F\in \bM_2$ having the resolution of type (2) (resp. type (3)) of Proposition \ref{resol}. Both of them are isomorphic to the complete linear system $|\cO_Q(2,2)|\cong \PP^8$ because the isomorphism class of sheaves is completely determined by the support. Note that $D_{1,0}\cup D_{0,1}\subset \bM_2^{s}$.
\end{defi}

\begin{rema}
By using the language of derived categories, a point in $D_{1,0}$ or $D_{0,1}$ can be described by an extension class. If $F\in D_{1,0}$ (similarly for $D_{0,1}$), then $F$ fits into an exact triangle
\[
	\ses{\cO_{Q}(1,0)}{F}{\cO_{Q}(-1,-2)[1]}.
\]
Conversely, for every nontrivial triangle of such form, $F$ is in the locus $D_{1,0}$ because $\Ext^1(\cO_{Q}(-1,-2)[1],\cO_{Q}(1,0)) =  \Ext^0(\cO_{Q}(-1,-2),\cO_{Q}(1,0))=H^0(\cO_{Q}(2,2))=\CC^9$. These objects form one of the walls in the sense of Bridgeland. After the Bridgeland wall crossing, we have a unique nontrivial triangle
\[
	\ses{\cO_{Q}(-1,-2)[1]}{G}{\cO_{Q}(1,0)}
\]
because $\Ext^1(\cO_{Q}(1,0),\cO_{Q}(-1,-2)[1])\cong\Ext^2(\cO_{Q}(1,0),\cO_{Q}(-1,-2))=\CC$. Thus the wall crossing is a (divisorial) contraction. Compare with the $\PP^2$ case in \cite[\S 6]{BMW14} and \cite[\S 6]{Mar13}. In Section \ref{sec:M2andR}, we will describe the contraction in terms of FM transforms (\cite[\S 4]{LP93a}).
\end{rema}


\section{Rationality of $\bM_{2}$}\label{sec:rationality}

In this short section, we answer the rationality question. Later in Sections \ref{sec:M2andR} and \ref{sec:RandK}, we will provide a more theoretical description of the birational map. Here we leave a very elementary proof of the rationality of $\bM_{2}$ for readers who are interested in the description of the birational map in classical terms.

\begin{prop}\label{birationalKM2}
Two spaces $\bK$ and $\bM_2$ are birationally equivalent.
\end{prop}
\begin{proof}
It is sufficient to construct an injective map on some Zariski dense open subset of the space $\bK$ since both varieties have the same dimension and are irreducible. We begin with $\bK_{1} := \bM_{0,1}(Gr(2, 4), 2L)$, the moduli space of 1-pointed stable maps. Let $\bK_{1}^{0}$ be the open subset parametrizing smooth 1-pointed conics in $Gr(2, 4)$ such that 1) it generates a smooth quadric in $\PP(V) = \PP^{3}$ and 2) the smooth quadric is not equal to $Q$. Let $\bK^{0} := \pi(\bK_{1}^{0})$ for the forgetful map $\pi$. For each stable map $[f: (\PP^{1}, x) \to Gr(2, 4)] \in \bK_{1}^{0}$, we have a pair $(Q', \ell)$ of a quadric surface in $\PP^{3}$ and a line $\ell \subset Q'$ obtained by a pointed conic in $Gr(2, 4)$. Clearly, it is well-defined on families.

Define $C:=Q'\cap Q$. Then $C$ is a quartic curve since $Q' \ne Q$. Let $\{p,q\}=\ell\cap Q$ be the set of intersection points. Then the line bundle $\cL:=\cO_C(p+q)$ on $C$ is a stable sheaf on $Q$ with Hilbert polynomial $4m+2$. Therefore we have a map $\psi : \bK_{1}^{0} \to \bM_{2}$.

A different choice of a marked point $x \in \PP^{1}$ defines a different line $\ell$ in the same ruling class. But the line bundle $\cL$ does not depend on the choice of $x$, because the natural map $\PP^{1} \to \mathrm{Pic}^{2}(C)$ is constant, since the image is a smooth elliptic curve. Therefore the map descends to a map $\bar{\psi} : \bK^{0} \to \bM_{2}$.

Now we show that $\bar{\psi}$ is injective. Let $F_1:=\cO_{C_1}(p_1+q_1)\cong \cO_{C_2}(p_2+q_2)=: F_2$ where $C_i=Q\cap Q_i$ and $\{p_i, q_i\}=Q\cap Q_i$ for $i=1,2$. Also let $\ell_{i} = \langle p_{i}, q_{i}\rangle$. Obviously $C_1=C_2$, so let $C := C_{1} = C_{2}$. Then $\langle Q,Q_1,Q_2\rangle \subset H^0(I_C(2))=\CC^2$. So the set $\{Q, Q_1, Q_2\}$ is linearly dependent. Without loss of generality, we may assume that
\begin{equation}\label{eqn:linearcombinationquadrics}
	Q_1= a Q_2+ b Q.
\end{equation}
Note that $a\ne 0$ because $Q_{1} \ne Q$. If $b\neq 0$, we have $\{p_1, q_1\}\subset Q_2$. If $\ell_1=\langle p_1,q_1\rangle\subset Q_2$, then $\ell_1\subset Q$ since $\ell_{1} \subset Q_{1}$ and \eqref{eqn:linearcombinationquadrics}. This is a contradiction to the fact that $\ell_1\cap Q =\{p_1,q_1\}$. So $\ell_1\cap Q_2 =\{p_1,q_1\}$ and $\ell_1$ is not a ruling of $Q_2$. Now, by using their different resolutions of the sheaves $F_i$ on $Q_{2}$ (see Proposition \ref{resol}), one can check that $F_1\ncong F_2$, which makes a contradiction.

Thus $b=0$ and $Q_1=Q_2$. By Proposition \ref{resol} again, if $\ell_{1}$ and $\ell_{2}$ are on the different ruling classes, then $F_{1}\ncong F_{2}$. Thus $\ell_{1}$ and $\ell_{2}$ are on the same ruling class and $(Q_{1}, \ell_{1})$ and $(Q_{2}, \ell_{2})$ come from the same stable map.
\end{proof}

As a consequence, we obtain the following result, which gives an affirmative answer for \cite[Conjecture 35]{BH14}.

\begin{coro}\label{cor:rationality}
The moduli space $\bM_{2}$ is rational.
\end{coro}

\begin{proof}
The moduli space $\bK$ is birational to $\bH := \mathrm{Hilb}^{2m+1}(Gr(2, 4))$, the Hilbert scheme of conics in $Gr(2, 4)$. The latter space is birational to $Gr(3, 6)$ (\cite[Section 3.10]{NT90}), which is a rational variety.
\end{proof}


\section{Birational map between $\bM_{2}$ and $\bR$ via Fourier-Mukai transforms}\label{sec:M2andR}

\subsection{The geometry of the moduli space $\bR$}
Recall that $\bR$ is the moduli space of semistable sheaves in $\PP^3$ with Hilbert polynomial $m^2+3m+2$. Here we collect well-known properties of $\bR$ intensively studied in \cite[Section 3]{LP93b}.

\begin{prop}\cite[Proposition 3.6, Remark 3.8]{LP93b}\label{propertyofr}
\begin{enumerate}
\item Every semistable sheaf $F\in \bR$ has a free resolution
\[
	\ses{2\cO_{\PP^{3}}(-1)}{2\cO_{\PP^{3}}}{F}.
\]
For the converse, see Lemma \ref{freeofr}.
\item
\begin{enumerate}
\item If $F\in \bR^s$, then $F\cong I_{L,S}(1)$ for an irreducible quadric surface $S$ and a line $L \subset S$.
\item If $F\in \bR\setminus \bR^s$, then $\mathrm{gr}(F)=\cO_H\oplus \cO_{H'}$ for planes $H$ and $H'$ in $\PP^3$.
\end{enumerate}
\item The map $\pi:\bR\longrightarrow |\cO_{\PP^{3}}(2)|\cong \PP^9$ defined by the Fitting ideal of $F$ is a double covering ramified along the discriminant divisor $\Delta$. Here, $\Delta$ denotes the locus of singular quadric surfaces in $\PP^3$, which is a degree 4 singular hypersurface.
\end{enumerate}
\end{prop}

We believe that the following result has also been observed by Le Potier. Since we could not find the precise statement and its proof, we leave it as a lemma.

\begin{lemm}\label{freeofr}
If $F \in \mathsf{Coh}(\PP^3)$ has a resolution
\[
	0 \rightarrow 2\cO_{\PP^{3}}(-1) \stackrel{M}{\rightarrow}
	2\cO_{\PP^{3}} \rightarrow F \rightarrow 0,
\]
then $F$ is semistable.
\end{lemm}
\begin{proof}
Choose an injective homomorphism $\cO_{\PP^{3}}(-1)\subset 2\cO_{\PP^{3}}(-1)$. By composing with $M$, we have an injective homomorphism
\[
	0\lr \cO_{\PP^{3}}(-1)\lr 2\cO_{\PP^{3}}
\]
such that the cokernel is isomorphic to a twisted ideal sheaf $I_{L,\PP^3}(1)$ for a line $L$ or $\cO_H\oplus \cO_{\PP^{3}}$ for a plane $H$.

\textsf{Case 1.} The cokernel is $I_{L,\PP^{3}}(1)$.

By applying the snake lemma to
\begin{equation}\label{eqn:twocomplexes}
	\xymatrix{0 \ar[r] & \cO_{\PP^{3}}(-1) \ar[r] \ar[d]
	& 2\cO_{\PP^{3}} \ar@{=}[d] \ar[r] & I_{L, \PP^{3}}(1)
	\ar[d] \ar[r] & 0\\
	0 \ar[r] & 2\cO_{\PP^{3}}(-1) \ar[r]^-{M} & 2\cO_{\PP^{3}} \ar[r]
	& F \ar[r] & 0,}
\end{equation}
we have an exact sequence
\[
	\ses{\cO_{\PP^{3}}(-1)}{I_{L,\PP^3}(1)}{F}.
\]
From $\ses{I_{L,\PP^3}(1)}{\cO_{\PP^{3}}(1)}{\cO_L(1)}$, by using the snake lemma again, one can show that $F$ is the kernel of the canonical surjection $\cO_{Q'}(1)\lr \cO_L(1)$ for a quadric surface $Q'$ determined by the inclusion $\cO_{\PP^{3}}(-1) \to I_{L, \PP^{3}}(1) \to \cO_{\PP^{3}}(1)$. Therefore $F=I_{L,Q'}(1)$. By Proposition \ref{propertyofr}, $F$ is semistable.

\textsf{Case 2.} The cokernel is $\cO_{H}\oplus \cO_{\PP^{3}}$.

By a similar diagram chasing, we have a short exact sequence
\[
	\ses{\cO_H}{F}{\cO_H'}
\]
for two planes $H$ and $H'$. Hence, $F$ is semistable because $\mathrm{gr}(F)=\cO_H\oplus \cO_H'$.
\end{proof}

So we have the following result.

\begin{prop}\label{prop:RandN}\cite[Corollary 3.7]{LP93b}
The space $\bR$ is isomorphic to the space $\bN(4;2,2)$ of Kronecker modules.
\end{prop}

By definition, the space $\bN(4; 2, 2)$ is constructed as the GIT quotient
\[
	\PP\Hom(\cO_{\PP^{3}}(-1)^{2}, \cO_{\PP^{3}}^{2})\git G
	\cong \PP\Hom(\CC^{2}\otimes V, \CC^{2})\git G
	\cong \PP(V^{*}\otimes \gl_{2})\git G,
\]
where $G = \PGL_{2} \times \PGL_{2}$ acts as $(A, B)\cdot M = AMB^{-1}$ (\cite[Corollary 3.7]{LP93b}). Therefore $\bR \cong \bN(4;2,2)$ is a normal variety. For this action, there are strictly semistable points and the GIT quotient is singular. In the next section, we will discuss a systematic (partial) resolution of these singularities, the so-called Kirwan's paritial desingularization (\cite{Kir85}).

As a simple corollary of Proposition \ref{propertyofr}, we have:

\begin{coro}
The canonical divisor $K_{\bR}$ is $-8\pi^*\cO_{\PP^9}(1)$.
\end{coro}

\begin{proof}
The singular locus $\mathrm{Sing}(\Delta)$ of $\Delta$ is isomorphic to $\PP^3\times \PP^3/\ZZ_2$. Thus the codimension of $\mathrm{Sing}(\Delta)$ in $\Delta$ is two. Hence if we apply the covering formula for the map $\pi$, then we have
\[
	K_{\bR}=\pi^*K_{\PP^9}+\frac{1}{2}[\Delta].
\]
Since $[\Delta]=4\pi^*\cO_{\PP^9}(1)$, we have the result.
\end{proof}

\subsection{A divisorial contraction of $\bM_2$}

In this subsection, by applying the method developed in \cite[Section 4]{LP93a}, we construct a birational morphism $\Psi : \bM_2 \longrightarrow \bR$ by using FM transforms.

Let $\Delta_1$ be the universal conic in $Q\subset \PP(V) = \PP^{3}$. Let $\Delta_2$ be the space of the universal planes in $(\PP^{3})^*$. We have two diagrams:
\[
	\xymatrix{ \Delta_1\subset (\PP^3)^*\times Q\ar[r]^-{p}\ar[d]_{q}&
	Q\\
	(\PP^3)^* \cong |\cO_Q(1,1)|}
\]
and
\[
	\xymatrix{ \Delta_2\subset \PP^3\times (\PP^3)^*\ar[r]^-{r}\ar[d]_{s}
	&(\PP^3)^*\\
	\PP^3,}
\]
where $p,q,r,s$ indicate the projection maps onto each factor.

\begin{defi}\label{def:FMtransforms}
Let
\[
	\Psi_1(F):=q_*(\cO_{\Delta_1}\otimes p^*(F\otimes \cO_{Q}(1,1)))
	=: F'
\]
and
\[
	\Psi_2(F'):=R^2 s_*(\cO_{\Delta_2}\otimes
	r^*(F'\otimes \cO_{(\PP^{3})^{*}}(-3))).
\]
\end{defi}

\begin{rema}
As we will see below in the the proof of Theorem \ref{mainthm1}, for those two transforms the other higher direct image sheaves vanish and thus the $\Psi_i$'s can be regraded as the FM transforms with the kernel $\cO_{\Delta_i}$ on derived categories.
\end{rema}

Now we prove the first main theorem of this paper.

\begin{theo}\label{mainthm1}
There exists a birational morphism
\[
	\Psi: \bM_2 \longrightarrow \bR
\]
given by $F\mapsto \Psi(F):=\Psi_2( \Psi_1 (F))$. The map $\Psi$ contracts two divisors $D_{1,0}$ and $D_{0,1}$ (Definition \ref{def:Ddivisor}) to two points $\Psi(D_{1,0})=\{[\cO_Q(1,0)]\}$, $\Psi(D_{0,1})=\{[\cO_Q(0,1)]\}$. Furthermore, $\Psi$ is a smooth blow up of two points $[\cO_{Q}(1,0)]$ and $[\cO_{Q}(0,1)]$.
\end{theo}

\begin{proof}
\textsf{Step 1.} For any $F\in \bM_{2}$ with the resolution of type (1) of Proposition \ref{resol}, $\Psi(F) \in \bR$.

\medskip

If $F$ has the resolution of type (1) of Proposition \ref{resol}, by tensoring $\cO_{Q}(1,1)$, we have
\begin{equation}\label{eq:twistedresol}
	\ses{2\cO_{Q}}{2\cO_{Q}(1,1)}{F(1,1)}.
\end{equation}
First we compute the direct image sheaves $R^i q_*(\cO_{\Delta_1}\otimes p^*\cO_{Q})$ and $R^i q_*(\cO_{\Delta_1}\otimes p^*\cO_{Q}(1,1))$. The resolution of $\cO_{\Delta_1}$ is given by
\begin{equation}\label{eq:resolofD1}
	\ses{\cO_{(\PP^{3})^{*} \times Q}(-1,(-1,-1))}
	{\cO_{(\PP^{3})^{*}\times Q}}{\cO_{\Delta_1}}.
\end{equation}
By tensoring $q^*\cO_{(\PP^{3})^{*}}$ and taking the direct image functor $R^{\bullet}p_*$, we obtain a long exact sequence
\[
	0\lr q_*\cO_{(\PP^{3})^{*}\times Q}(-1,(-1,-1)) \lr
	q_*\cO_{(\PP^{3})^{*}\times Q} \lr
	q_*(\cO_{\Delta_1}\otimes p^*\cO_{(\PP^{3})^{*}\times Q})
\]
\[
	\lr R^1 q_*\cO_{(\PP^{3})^{*}\times Q}(-1,(-1,-1))\lr \cdots.
\]
But $H^i(\cO_{Q}(-1,-1))=H^j(\cO_{Q})=0$ for all $i$ and $j\geq1$. Also, $H^0(\cO_{Q})=\CC$. Thus we obtain that $R^i q_*\cO_{(\PP^{3})^{*}\times Q}(-1,(-1,-1))=0$ for all $i$ and $R^j q_*\cO_{(\PP^{3})^{*}\times Q}=0$ for all $j\geq1$. Also $q_*\cO_{(\PP^{3})^{*}\times Q}\cong \cO_{(\PP^{3})^{*}}$. Therefore, we have $q_*(\cO_{\Delta_1}\otimes p^*\cO_{Q})\cong \cO_{(\PP^{3})^{*}}$ and $ R^i q_*(\cO_{\Delta_1}\otimes p^*\cO_{Q})=0$ for all $i\geq1$.

One can similarly compute $R^i q_*(\cO_{\Delta_1}\otimes p^*\cO_{Q}(1,1))$. From the resolution \eqref{eq:resolofD1} of $\cO_{\Delta_1}$, we have
\[
	0\lr q_*\cO_{(\PP^{3})^{*}\times Q}(-1,(0,0)) \lr
	q_*\cO_{(\PP^{3})^{*}\times Q}(0,(1,1)) \lr
	q_*(\cO_{\Delta_1}\otimes p^*\cO_{Q}(1,1))
\]
\[
	\lr  R^1 q_*\cO_{(\PP^{3})^{*}\times Q}(-1,(0,0))\lr \cdots.
\]
But $H^i(\cO_{Q})=\CC$ for $i=0$ and $0$ otherwise. Also, $H^j(\cO_{Q}(1,1))=\CC^4$ for $j=0$ and $0$ otherwise. Hence $R^i  q_*\cO_{(\PP^{3})^{*}\times Q}(-1,(0,0))\cong \cO_{(\PP^{3})^{*}}(-1)$ for $i=0$ and $0$ otherwise. Also, $R^j q_*\cO_{(\PP^{3})^{*}\times Q}(0,(1,1))=4\cO_{(\PP^{3})^{*}}$ for $j=0$ and $0$ otherwise.

By applying these two computations to \eqref{eq:twistedresol}, we have a resolution
\[
	\ses{\cO_{(\PP^{3})^{*}}(-1)}{4\cO_{(\PP^{3})^{*}}}
	{q_*(\cO_{\Delta_1}\otimes p^*\cO_{Q}(1,1))=:\cQ}.
\]
Therefore, by taking the FM-transform $\Psi_1$, we obtain a short exact sequence
\[
	\ses{2\cO_{(\PP^{3})^{*}}}{2\cQ_{(\PP^{3})^{*}}}{\Psi_1(F)}.
\]
Note that $R^i q_*(\cO_{\Delta_1}\otimes p^*F(1,1))=0$ for all $i\geq1$.

Take the second FM-transform for $F' :=\Psi_1(F)$. By a similar computation using the resolution of $\cO_{\Delta_2}$ and the Serre duality, we have
\[
	\Psi_2(\cO_{(\PP^{3})^{*}})\cong \cO_{\PP^{3}}(-1), \quad
	\Psi_2(\cQ)\cong\cO_{\PP^{3}}.
\]
Furthermore, $R^j s_*(\cO_{\Delta_2}\otimes r^*(F'(-3)))=0$ for $j\neq2$. Hence from the above short exact sequence, we obtain a resolution
\[
	\ses{2\cO_{\PP^{3}}(-1)}{2\cO_{\PP^{3}}}{\Psi_2(F')}.
\]
By Lemma \ref{freeofr}, the sheaf $\Psi_2(F')$ is semistable with Hilbert polynomial $m^2+3m+2$.

\medskip

\textsf{Step 2.} For $F \in \bM_{2}$ with a resolution of the type (2) or (3) in Proposition \ref{resol}, $\Psi(F) \in \bR$.

\medskip

If $F\in D_{1,0}$, then $F\cong\cO_C(1,0)$ for some quartic curve $C$. From the resolution
\[
	\ses{\cO_{Q}(0,-1)}{\cO_{Q}(2,1)}{F(1,1)},
\]
it is straightforward to check that
\[
	\Psi_1(F)=q_*(\cO_{\Delta_1}\otimes p^*\cO_{Q}(2,1)).
\]
by using a similar computation in the previous step. Note that the image $\Psi_1(F)$ does not depend on the choice of the quartic curve $C$.

Let $i : Q \hookrightarrow \PP^{3}$ be the inclusion and let $j : (\PP^{3})^{*} \times Q \to \PP^{3} \times (\PP^{3})^{*}$ be the map defined by $j(x,y) = (y, i(x))$. Then
\begin{align*}
\Psi_2(\Psi_1(F)) &= R^2 s_*(\cO_{\Delta_2}\otimes r^*(\Psi_1(F)\otimes \cO_{(\PP^{3})^{*}}(-3))) \\
& \cong R^2 s_*(\cO_{\Delta_2}\otimes r^*(q_*(\cO_{\Delta_1}\otimes p^*\cO_{Q}(2,1))\otimes \cO_{(\PP^{3})^{*}}(-3)))\\
& \cong R^2 s_*(\cO_{\Delta_2}\otimes r^*((r_* j_*)(\cO_{\Delta_1}\otimes p^*\cO_{Q}(2,1))\otimes r^*\cO_{(\PP^{3})^{*}}(-3)))\\
& \cong R^2 s_*(\cO_{\Delta_2}\otimes r^*(r_*(\cO_{\Delta_2}\otimes j_* p^*\cO_{Q}(2,1))\otimes \cO_{(\PP^{3})^{*}}(-3)))\\
& \cong R^2s_*(\cO_{\Delta_2}\otimes r^*(r_*(\cO_{\Delta_2}\otimes s^*(i_{*}\cO_{Q}(2,1)))\otimes \cO_{(\PP^{3})^{*}}(-3)))=: G.
\end{align*}
The third isomorphism comes from $q=r\circ j$ and the fifth isomorphism comes from $s \circ j= i \circ p$. This can be regarded as a composition of two FM-transforms of $i_{*}\cO_{Q}(2,1)$ with the same kernel $\cO_{\Delta_2}$.

Let
\begin{equation}\label{eq:resolofO21}
	0 \rightarrow 2\cO_{\PP^{3}} \stackrel{M}{\rightarrow}
	2\cO_{\PP^{3}}(1) \rightarrow i_{*}\cO_{Q}(2,1) \rightarrow 0
\end{equation}
be a resolution of $i_{*}\cO_{Q}(2,1)$ on $\PP^3$. By using the resolution of $\cO_{\Delta_2}$, we obtain a short exact sequence on $(\PP^3)^*$
\[
	\ses{2\cO_{(\PP^{3})^{*}}}{2\cQ''_{(\PP^{3})^{*}}}
	{r_*(\cO_{\Delta_2}\otimes s^*(i_{*}\cO_{Q}(2,1)))},
\]
where $\cQ'':=\mathrm{coker}(\cO_{(\PP^{3})^{*}}(-1)\hookrightarrow 4\cO_{(\PP^{})^{*}})$. By tensoring $\cO_{(\PP^{3})^{*}}(-3)$ and taking the direct image functor $R^{\bullet}s_*$, we have a short exact sequence
\[
	\ses{2\cO_{\PP^{3}}(-1)}{2\cO_{\PP^{3}}}{G}.
\]
By Lemma \ref{freeofr}, $G \in \bR$. One can also check that $R^i s_*(-)$ for $i\neq 2$ vanish.

Furthermore, we can show that $G \cong \cO_{Q}(1,0)$. Indeed, the map $M$ in \eqref{eq:resolofO21} is completely recovered (up to a twisting) after taking two FM-transforms and thus the cokernel is as well. This implies that $G(1)=i_{*}\cO_{Q}(2,1)$. Hence $\Psi_2(\Psi_1(F))\cong \cO_{\PP^{3}}(-1)\otimes i_{*}\cO_{Q}(2,1)=i_{*}\cO_{Q}(1,0)$.

The case of $F \in D_{0,1}$ can be computed in a similar way.

\medskip

\textsf{Step 3.} The map $\Psi$ is birational.

\medskip

The inverse map $\Psi_1^{-1}\circ \Psi_2^{-1}:\bR \setminus\{[\cO_Q(1,0)], [\cO_Q(0,1)]\} \longrightarrow \bM_2 \setminus D_{1,0}\cup D_{0,1}$ can be defined by \emph{inverse} FM-transforms
\[
	\Psi_1^{-1}(F'):=R^2p_*(\cO_{\Delta_1}\otimes q^*(F'(-3))), \quad
	\Psi_2^{-1}(F''):=r_*(\cO_{\Delta_2}\otimes s^*(F''(1))).
\]
By using the standard resolution of $F'' \in \bR$ in Proposition \ref{propertyofr} and the same computation, one can show that $\Psi_2^{-1}(F'')=F'$ and $\Psi_1^{-1}(F')=F$ by the same argument in the proof of the claim before. Hence $\Psi$ is injective and thus $\Psi$ is a birational morphism.

\medskip

\textsf{Step 4.} The map $\Psi$ is a smooth blow up.

\medskip

Finally, we show that the map $\Psi$ is a smooth blow up.
From the standard deformation theory on the moduli space of sheaves, at the point $[\cO_Q(1,0)]\in \bR$, the tangent space of $\bR$ is canonically isomorphic to $\Ext^{1}_{\PP^{3}}(\cO_{Q}(1,0), \cO_{Q}(1,0))$. We show that $T_{[\cO_Q(1,0)],\bR} \cong H^0(\cO_{Q}(2,2))=\CC^9$. Consider the following long exact sequence (\cite[Lemma 13]{CCM14}):
\[
	0\lr \Ext_Q^1(\cO_Q(1,0),\cO_Q(1,0))\lr \Ext_{\PP^3}^1(\cO_Q(1,0),	\cO_Q(1,0))\lr
\]
\[
	\Ext_Q^0(\cO_Q(1,0),\cO_Q(1,0)\otimes N_{Q/\PP^3})\lr 		
	\Ext_Q^2(\cO_Q(1,0),\cO_Q(1,0))\lr\cdots
\]
But $\Ext_Q^i(\cO_Q(1,0),\cO_Q(1,0))=0$ for $i\geq1$. So
\[
	\Ext_{\PP^3}^1(\cO_Q(1,0),\cO_Q(1,0))\cong \Ext_Q^0(\cO_Q(1,0),	\cO_Q(1,0)\otimes N_{Q/\PP^3}),
\]
where the latter space is isomorphic to $\Ext_Q^0(\cO_Q(1,0), \cO_Q(1,0)\otimes N_{Q/\PP^3})\cong H^0(\cO_Q(2, 2))$.

On the other hand, by definition (see Proposition \ref{resol} and Definition \ref{def:Ddivisor}), the sheaves on the divisor $D_{1,0}$ can be parametrized by
\[
	i:\PP\Hom(\cO_{Q}(-1,-2),\cO_{Q}(1,0)) =
	\PP(T_{[\cO_Q(1,0)],\bR})\hookrightarrow \bM_2.
\]
Hence, we obtain a commutative diagram
\[
	\xymatrix{ D_{1,0}\ar@{^(->}[r]^{i}\ar[d]&\bM_2\ar[d]^{\Psi}\\
	\{[\cO_Q(1,0)]\}\ar@{^(->}[r]&\bR.}
\]
and thus the morphism $\Psi$ is a smooth blow up morphism of $\bR$ at the point $[\cO_Q(1,0)]$. The case of $[\cO_{Q}(0,1)]$ is identical.
\end{proof}

\begin{coro}
The canonical divisor $K_{\bM_{2}}$ is given by $\Psi^{*}K_{\bR} + 8(D_{1,0}+D_{0,1})$.
\end{coro}

\begin{rema}
Recall that the virtual Euler number $e(\bM_{2})$ is 26 (Corollary \ref{virpoinm2}). By Theorem \ref{mainthm1} and the blow up formula of the virtual Poincar\'e polynomial, one can easily see that $e(\bR)=10$. This can be explained in the following different way. The virtual Euler number of the stable part $\bR^s$ is $e(\bR^s)=0$ (\cite[Corollary 6.14]{Wei13}). But by Proposition \ref{propertyofr}, the complement $\bR\setminus \bR^s$ is isomorphic to the space $\mathrm{Sym}^2(\PP^3)$ parametrizing two planes in $\PP^3$. Thus $e(\bR)=e(\bR^s)+e(\mathrm{Sym}^2(\PP^3))=10$.
\end{rema}


\section{Birational map between $\bR$ and $\bK$ via partial desingularization}\label{sec:RandK}

Recall that the space $\bR$ can be constructed as the GIT quotient
\[
	\PP\Hom(\CC^{2}\otimes V, \CC^{2})\git G,
\]
where $G = \PGL_{2} \times \PGL_{2}$ acts as $(A, B)\cdot M = AMB^{-1}$. For this action, there are strictly semistable points and the GIT quotient $\bR$ is very singular. In this section, we study Kirwan's partial desingularization (\cite{Kir85}) of $\bR$, which is a systematic procedure to resolve these singularities in a $G$-equivariant way.

The next result is the second main theorem of this paper.

\begin{theo}\label{thm:pardesing}
The partial desingularization of $\bR$ is isomorphic to $\bK$, Kontsevich's moduli space of degree 2 stable maps to the Grassmannian $Gr(2, 4)$.
\end{theo}

We give the proof in Section \ref{ssec:elemmod}.

\subsection{A description of the birational map}

The birational map $\bR \dashrightarrow \bK$ can be described in the following way. Let
\[
	\cO_{S \times \PP^{3}}(-1)^{2} \stackrel{M}{\to}
	\cO_{S \times \PP^{3}}^{2}
\]
be a family of Kronecker modules over $S$. Consider the projectivization $\PP(\cO_{S \times \PP^{3}}(-1)^{2}) \cong \PP(\cO_{S \times \PP^{3}}^{2}) \cong S \times \PP^{3} \times \PP^{1}$ and let $p : S \times \PP^{3} \times \PP^{1} \to S \times \PP^{1}$ and $q : S \times \PP^{3} \times \PP^{1} \to S \times \PP^{3}$ be two projections. By taking the pull-back, we have
\[
	\cO_{S \times \PP^{3} \times \PP^{1}}(-1, 0)^{2}
	\stackrel{q^{*}M}{\longrightarrow}
	\cO_{S \times \PP^{3} \times \PP^{1}}^{2}.
\]
Let $\iota : \cO_{S \times \PP^{3} \times \PP^{1}}(-1, -1) \hookrightarrow \cO_{S \times \PP^{3} \times \PP^{1}}(-1, 0)^{2}$ be the tautological subbundle. By taking the dual, we have
\[
	\cO_{S \times \PP^{3} \times \PP^{1}}^{2}
	\stackrel{(q^{*}M \circ \iota)^{*}}{\longrightarrow}
	\cO_{S \times \PP^{3} \times \PP^{1}}(1, 1).
\]
By taking the push-forward $p_{*}$ and $\otimes \cO_{S \times \PP^{1}}(-1)$, we obtain
\[
	\cO_{S \times \PP^{1}}(-1)^{2}
	\stackrel{p_{*}(q^{*}M \circ \iota)^{*}\otimes \cO_{\PP^{1}}(-1)}
	{\longrightarrow} \cO_{S \times \PP^{1}}^{4}.
\]
Finally, take its dual again:
\begin{equation}\label{eqn:bundletomap}
	\cO_{S \times \PP^{1}}^{4}
	\stackrel{(p_{*}(q^{*}M \circ \iota)^{*}
	\otimes \cO_{\PP^{1}}(-1))^{*}}{\longrightarrow}
	\cO_{S \times \PP^{1}}^{2}(1).
\end{equation}
For a general fiber, it is an epimorphism to a rank 2, degree 2 bundle, so we have a degree 2 map $\PP^{1} \to \mathrm{Gr}(2, 4)$. Hence we have a rational family of stable maps
\[
	\Phi(M) : S \dashrightarrow \bK,
\]
which is regular on the locus that $(p_{*}(q^{*}M \circ \iota)^{*}\otimes \cO_{\PP^{1}}(-1))^{*}$ is surjective.

Therefore there is a rational map
\begin{equation}\label{eqn:KroneckertoKontsevich}
	\Phi : \PP\Hom(\CC^{2} \otimes V, \CC^{2})^{ss}
	\dashrightarrow \bK.
\end{equation}
It is straightforward to see that $\Phi$ is $G$-invariant, because the $G$-action simply changes the coordinates of the fiber $\PP^{1}$ and the automorphism of $\cO_{\PP^{1}}(-1)^{2} \hookrightarrow \cO_{\PP^{1}}^{4}$. Therefore we have the quotient map
\[
	\overline{\Phi} : \bR \dashrightarrow \bK.
\]

\subsection{GIT stability of the moduli space of Kronecker modules}

The GIT (semi)stability of the moduli space of Kronecker modules is already well-known. For the proof, consult \cite[Proposition 15]{JMd87}. Also see \cite[Remark 4.9]{ACK07}.

\begin{theo}
A closed point $M \in \PP\Hom(A \otimes V, B)$ is (semi)stable with respect to the $G$-action if and only if for every nonzero proper subspace $A' \subset A$ and $B' \subset B$ such that $M(A' \otimes V) \subset B'$,
\[
	\frac{\dim B}{\dim A} \; (\le) < \frac{\dim B'}{\dim A'}.
\]
\end{theo}

In our special case that $\dim A = \dim B = 2$, we have:

\begin{coro}
Let $\bX = \PP\Hom(\CC^{2} \otimes V, \CC^{2}) \cong \PP(V^{*} \otimes \gl_{2})$ with the prescribed linearized $G$-action. $M \in \bX$ is (semi)stable if and only if for every one-dimensional subspace $A' \subset \CC^{2}$, $\dim \mathrm{im} \;M(A' \otimes V) (\ge) > 1$.
\end{coro}

This stability condition can be described in down-to-earth terms. We can describe $M \in \bX$ as a $2 \times 2$ matrix of linear polynomials with four variables $x, y, z, w$. If $M$ is semistable, then even after performing row/column operations, there is no zero row or column. If $M$ is stable, then even after any row/column operations, $M$ has no zero entry. Therefore we have the following results.

\begin{lemm}
Let $M \in \bX$ be a closed point.
\begin{enumerate}
\item If $M$ is unstable, then
\[
	M \in G \cdot \left[\begin{array}{cc}g & h\\0&0\end{array}\right]
	\mbox{ or }
	M \in G \cdot \left[\begin{array}{cc}g &0\\h&0\end{array}\right]
\]
for some $g, h \in V^{*}$.
\item If $M \in \bX^{ss} - \bX^{s}$, then
\[
	M \in G \cdot \left[\begin{array}{cc}g&k\\0&h\end{array}\right]
\]
for some $g, h \in V^{*}\setminus \{0\}$ and $k \in V^{*}$.
\item If $M \in \bX^{ss} - \bX^{s}$ and $M$ has a closed orbit in $\bX^{ss}$, then $k = 0$ and
\[
	M \in G \cdot \left[\begin{array}{cc}g & 0 \\ 0 & h\end{array}\right]
\]
for some $g, h \in V^{*}\setminus \{0\}$. If $g = h$, then $\Stab M \cong \PGL_{2}$. If not, $\Stab M \cong \CC^{*}$.
\end{enumerate}
\end{lemm}

In particular, we obtain:

\begin{lemm}\label{lem:regularonstablepart}
The rational map $\Phi$ in \eqref{eqn:KroneckertoKontsevich} is regular on $\bX^{s}$.
\end{lemm}

\begin{proof}
By the above construction, for each $M \in \bX^{s}$, $\Phi(M)$ is given by a rank 2 quotient bundle
\[
	\Phi(M) : \cO_{\PP^{1}}^{4} \to \cO_{\PP^{1}}(1)^{2}
\]
such that for each $p \in \PP^{1}$, the map is given by a linear combination of coefficients of two rows of $M$. Since $M \in \bX^{s}$, every linear combination has full rank. Therefore $\Phi(M)$ is a surjection and it defines a stable map.
\end{proof}

\subsection{A stratification on $\bX^{ss}$}

Now we define a stratification
\[
	\bX^{ss} = \bY_{0} \sqcup \bZ_{0} \sqcup \bY_{1} \sqcup \bZ_{1}
	\sqcup \bX^{s}
\]
as the following. The last open stratum $\bX^{s}$ is the stable locus.

Let $\bY_{0}'$ be the image of
\[
	\rho_{0} : \PP(\mathfrak{gl}_{2}) \times \PP(V^{*}) \to \bX, \quad
	(A, g) \mapsto A \left[\begin{array}{cc}
	g & 0\\0 & g\end{array}\right]
\]
and let $\bY_{0} := \bY_{0}' \cap \bX^{ss}$. Then $\bY_{0} = \rho_{0}(\PGL_{2} \times \PP(V^{*}))$ and on this locus $\rho_{0}$ is an embedding. So $\bY_{0}$ is a smooth closed subvariety of dimension 6. Finally, at each closed point $M \in \bY_{0}$, the normal bundle $N_{\bY_{0}/\bX^{ss}}|_{M}$ is naturally isomorphic to $H \otimes \mathfrak{sl}_{2}$, where $H$ is a 3-dimensional quotient space of $V^{*}$.

Let $\bZ'_{0}$ be the image of
\[
	\tau_{0} : G \times \PP(V^{*}\oplus V^{*}) \times \PP^{2} \to \bX,
	\quad ((A, B), (g, h), [s:t:u]) \mapsto
	A\left[\begin{array}{cc}sg & uh\\ 0 & tg\end{array}\right]B^{-1}.
\]
Let $\bZ_{0} := (\bZ'_{0} \cap \bX^{ss}) \setminus \bY_{0}$ and let $\overline{\bZ}_{0} := \bZ_{0} \sqcup \bY_{0}$, the closure of $\bZ_{0}$ in $\bX^{ss}$. By computing the dimension of a general fiber, one can see that $\bZ_{0}$ is a 10-dimensional irreducible $G$-invariant variety. The normal cone $C_{\bY_{0}/\overline{\bZ}_{0}}$ is an analytic locally trivial bundle, whose fiber at $M \in \bY_{0}$ is isomorphic to $\Stab M \cdot (H \otimes \langle e\rangle) = \Stab M \cdot (H \otimes \langle f\rangle)$ where $\{h, e, f\}$ is the standard basis of $\mathfrak{sl}_{2}$. In the projectivized normal space $\PP(H \otimes \mathfrak{sl}_{2})$, $\PP(C_{\bY_{0}/\overline{\bZ}_{0}}|_{M}) \cong \PP H \times \PGL_{2}\cdot \PP\langle e\rangle \cong \PP H \times \PP^{1}$. Thus $C_{\bY_{0}/\overline{\bZ}_{0}}$ is a degree 2 bundle of dimension 4.

Let $\bY_{1}'$ be the image of
\[
	\rho_{1} : G \times \PP (V^{*}\oplus V^{*}) \to \bX, \quad
	((A, B), (g, h)) \mapsto A\left[\begin{array}{cc}g&0\\0&h
	\end{array}\right]B^{-1}.
\]
Similarly, let $\bY_{1} := (\bY'_{1} \cap \bX^{ss}) \setminus \bY_{0}$ and let $\overline{\bY}_{1} := \bY_{0} \sqcup \bY_{1}$, the closure of $\bY_{1}$ in $\bX^{ss}$. A general fiber of $\rho_{1}$ has dimension 2, so $\bY_{1}$ has dimension 11. By computing its general fiber, we can conclude that $\rho_{1}|_{\rho_{1}^{-1}(\bY_{1})}$ is smooth and thus $\bY_{1}$ is smooth, too. $\overline{\bY}_{1}$ is singular precisely along $\bY_{0}$. For $M \in \bY_{1}$, the fiber of the normal bundle $N_{\bY_{1}/\bX^{ss}}|_{M}$ is naturally isomorphic to $K \otimes \langle e, f\rangle$, where $K$ is a 2-dimensional quotient space of $V^{*}$. Finally, the normal cone $C_{\bY_{0}/\overline{\bY}_{1}}$ is a cone over a smooth 4-dimensional variety in $N_{\bY_{0}/\bX^{ss}}$. Indeed, at $M \in \bY_{0}$, $C_{\bY_{0}/\bX^{ss}}|_{M} \cong \Stab M \cdot (H \otimes \langle h \rangle) \subset H \otimes \mathfrak{sl}_{2} \cong N_{\bY_{0}/\bX^{ss}}|_{M}$ and $\PP(C_{\bY_{0}/\bX^{ss}}|_{M}) \cong \PP H \times \PGL_{2} \cdot \PP\langle h \rangle \cong \PP H \times \PP^{2} \subset \PP(H \otimes \mathfrak{sl}_{2})$.

Finally, let $\bZ_{1}'$ be the image of
\[
	\tau_{1} : G \times \PP((V^{*})^{3}) \to \bX, \quad
	((A, B), (g, h, k)) \mapsto A\left[\begin{array}{cc}g&k\\0&h
	\end{array}\right]B^{-1}
\]
and let $\bZ_{1} := (\bZ_{1}' \cap \bX^{ss}) \setminus (\bY_{0} \sqcup \bZ_{0} \sqcup \bY_{1})$. And let $\overline{\bZ}_{1} = \bZ_{1} \sqcup \bY_{0} \sqcup \bZ_{0} \sqcup \bY_{1}$, the closure of $\bZ_{1}$ in $\bX^{ss}$. By checking the dimension of a general fiber, we can see that $\bZ_{1}$ is a codimension two irreducible $G$-invariant subvariety of $\bX^{ss}$. The normal cone $C_{\bY_{1}/\overline{\bZ}_{1}}$ is a fiber bundle, analytic locally the union of two transversal rank 2 subbundles of the normal bundle $N_{\bY_{1}/\bX^{ss}}$. For $M \in \bY_{1}$, $C_{\bY_{1}/\overline{\bZ}_{1}}|_{M} \cong K \otimes \langle e \rangle \cup K \otimes \langle f \rangle \subset K \otimes \langle e, f\rangle \cong N_{\bY_{1}/\bX^{ss}}|_{M}$.

\subsection{Partial desingularization - an outline}

In our situation, Kirwan's partial desingularization is obtained through the following. For the general statement and its proof, see \cite{Kir85}. Set $\bX^{0} := \bX^{ss}$. First of all, take the blow up $\pi_{1} : {\bX^{1}}' \to \bX^{0}$ along $\bY_{0}$, the deepest stratum with the largest stabilizer. Then ${\bX^{1}}'$ is a smooth variety with a $G$-action since $\bY_{0}$ is a smooth $G$-invariant subvariety. Let $\bY_{0}^{1}$ be the exceptional divisor and let $\overline{\bY}_{1}^{1}, \overline{\bZ}_{i}^{1}$ be the proper transforms of $\overline{\bY}_{1}$, $\overline{\bZ}_{i}$ respectively. Note that $\overline{\bY}_{1}^{1}$ is a smooth subvariety of ${\bX^{1}}'$, since the normal cone $C_{\bY_{0}/\overline{\bY}_{1}}$ is a cone over a smooth variety.

For a linearized $\QQ$-ample line bundle $L_{0}$ on $\bX^{0}$ (this is unique up to scaling, since $\mathrm{rank}\;\mathrm{Pic}(\bX) = 1$.), let $L_{1} := \pi_{1}^{*}L_{0}\otimes \cO(-\epsilon_{1}\bY_{0}^{1})$ for sufficiently small $\epsilon_{1} > 0$. Then $L_{1}$ is an ample $\QQ$-line bundle and induces a linearized $G$-action. With respect to this $G$-linearized action on ${\bX^{1}}'$, $\overline{\bZ}_{0}^{1}$ becomes unstable. Let $\bX^{1} := {\bX^{1}}' \setminus \overline{\bZ}_{0}^{1}$.

Let $\pi_{2} : {\bX^{2}}' \to \bX^{1}$ be the blow up along $\overline{\bY}_{1}^{1}$. Then ${\bX^{2}}'$ is a smooth variety with a $G$-action since $\overline{\bY}_{1}^{1}$ is a $G$-invariant smooth subvariety. Let $\overline{\bY}_{1}^{2}$ be the exceptional divisor and let $\bY_{0}^{2}$ (resp. $\overline{\bZ}_{1}^{2}$) be the proper transform of $\bY_{0}^{1}$ (resp. $\overline{\bZ}_{1}^{1}$). Let $L_{2} := \pi_{2}^{*}L_{1}\otimes \cO(-\epsilon_{2}\overline{\bY}_{1}^{2})$ for sufficiently small $0 < \epsilon_{2} \ll \epsilon_{1}$. Since $\overline{\bY}_{1}^{1}$ is $G$-invariant, there is a well-defined $G$-linearized action on $L_{2}$. With respect to this linearized action on ${\bX^{2}}'$, $\overline{\bZ}_{1}^{2}$ is unstable. Let $\bX^{2} := {\bX^{2}}' \setminus \overline{\bZ}_{1}^{2}$.

Now $\bX^{2} = (\bX^{2})^{ss} = (\bX^{2})^{s}$ because there is no semistable point with a positive dimensional stabilizer. The partial desingularization of $\bX\git G$ is defined by $\bX^{2}/G$. If we denote $\pi_{1} \circ \pi_{2}$ by $\pi$, then since
\[
	\pi^{-1}(\bX^{s}) \subset (\bX^{2})^{s} \subset (\bX^{2})^{ss}
	\subset \pi^{-1}(\bX^{ss})
\]
in general, there are quotient maps $\overline{\pi}_{i} : \bX^{i}/G \to \bX^{i-1}/G$ for $i = 1, 2$. In summary, we have:
\[
	\xymatrix{\bX^{2} \ar[r]^{\pi_{2}}_{\overline{\bY}_{1}^{1}}
	\ar[d]^{/G} &
	\bX^{1} \ar[r]^{\pi_{1}}_{\bY_{0}} \ar[d]^{/G}
	& \bX^{0} \ar[d]^{/G}\\
	\bX^{2}/G \ar[r]^{\overline{\pi}_{2}} &
	\bX^{1}/G \ar[r]^{\overline{\pi}_{1}} & \bX\git G.}
\]
Note that every point $M \in \bX^{2}$ has only finite stabilizers with respect to the $G$-action, thus $\bX^{2}/G$ has orbifold singularities only.

\subsection{Analysis on fibers}

Here we will take a look at the change in a fiber of two exceptional divisors. For $M \in \bY_{0}$, $\pi_{1}^{-1}(M) \cong \PP(H \otimes \mathfrak{sl}_{2})$ where $H$ is a 3-dimensional quotient space of $V^{*}$. On $\pi_{1}^{-1}(M)$, $\Stab M \cong \PGL_{2}$-action is induced by the trivial action on $V^{*}$ and the standard $\SL_{2}$-adjoint action on $\mathfrak{sl}_{2}$. With respect to this $\PGL_{2}$-action, the unstable locus is isomorphic to $\PP H \times \PP^{1}$, which is precisely $\PP(\cC_{\bY_{0}/\overline{\bZ}_{0}}|_{M})$. Thus in $\bX^{1}$, the inverse image of $M$ is $\PP(V \oplus \mathfrak{sl}_{2})^{ss}$. If we denote the image of $M$ in $\bX\git G$ by $\overline{M}$, then
\[
	\overline{\pi}_{1}^{-1}(\overline{M}) \cong
	\PP(V \otimes \mathfrak{sl}_{2})\git \Stab M \cong
	\PP(V \otimes \mathfrak{sl}_{2})\git \PGL_{2}.
\]

The strictly semistable locus is isomorphic to $\PP H \times \PP (\mathfrak{sl}_{2}) \setminus (\PP H \times \PP^{1})$, which is the projectivized normal cone $\PP (C_{\bY_{0}/\overline{\bY}_{1}}|_{M})$. Therefore over the fiber of $\overline{M}$, the second blow up $\overline{\pi}_{2} : \bX^{2}/G \to \bX^{1}/G$ is precisely the partial desingularization of the fiber $\PP(H \otimes \mathfrak{sl}_{2})\git \PGL_{2}$. This resolution is well-studied in \cite[Theorem 4.1]{Kie07}. The blown-up fiber is isomorphic to $\bM_{0,0}(\PP^{2}, 2)$, the moduli space of degree 2 stable maps in $\PP^{2}$. This is isomorphic to the blow up of $\PP^{5} = \PP(H \otimes \mathfrak{sl}_{2})\git \PGL_{2}$ along the degree 2 Veronese embedding of $\PP^{2}$.

For $M \in \bY_{1}^{1}$, $\pi_{2}^{-1}(M) = \PP(K \otimes \langle e, f\rangle)$. With respect to the $\Stab M \cong \CC^{*}$ action, the unstable locus is $\PP(K \otimes \langle e\rangle) \sqcup \PP(K \otimes \langle f\rangle)$. So on $\bX^{2}$, $\pi_{2}^{-1}(M) = \PP(K \otimes \langle e, f\rangle)^{s}$. Also in $\bX^{2}/G$, $\overline{\pi}_{2}^{-1}(\overline{M}) = \PP (K\otimes \langle e, f\rangle)\git \CC^{*} \cong \PP^{1} \times \PP^{1}$.

\subsection{Elementary modification of maps}\label{ssec:elemmod}

The two steps on the partial desingularization can be understood as two steps of the modification of a family of maps. In this section, we prove Theorem \ref{thm:pardesing}.

\begin{proof}[Proof of Theorem \ref{thm:pardesing}]
Let $\bX^{0} := \bX^{ss} = \PP(\Hom(\CC^{2} \otimes V, \CC^{2}))^{ss}$ as in previous sections. By \eqref{eqn:bundletomap} (applied to $S = \bX^{0}$), we have a rational map
\[
	f_{0} : \bX^{0} \times \PP^{1}\dashrightarrow Gr(2, 4)
\]
and by Lemma \ref{lem:regularonstablepart}, $f_{0}$ is regular on $\bX^{s} \times \PP^{1}$.

By the Pl\"ucker embedding, $Gr(2, 4)$ is embedded into $\PP^{5}$ as a quadric hypersurface. By composing with this embedding, we may regard $f_{0}$ as a family of maps to $\PP^{5}$. Indeed, this map is given by a bundle morphism
\begin{equation}\label{eqn:bundlemorphism}
	6\cO_{\bX^{0} \times \PP^{1}}
	= \wedge^{2}4\cO_{\bX^{0} \times \PP^{1}}
	\stackrel{\wedge^{2}(p_{*}(q^{*}M \circ \iota)^{*}
	\otimes \cO_{\PP^{1}}(-1))^{*}}{\longrightarrow}
	\wedge^{2}2\cO_{\bX^{0} \times \PP^{1}}(1)
	= \cO_{\bX^{0}\times \PP^{1}}(2).
\end{equation}
Let $F_{0} := \wedge^{2}(p_{*}(q^{*}M \circ \iota)^{*}\otimes \cO_{\PP^{1}}(-1))^{*}$.

Consider the blow up map $\pi_{1} \times \mathrm{id} : \bX^{1} \times \PP^{1} \to \bX^{0} \times \PP^{1}$. Then by taking the pull-back, we have a morphism of sheaves
\[
	6\cO_{\bX^{1}\times \PP^{1}}
	\stackrel{(\pi_{1}\times \mathrm{id})^{*}F_{0}}{\longrightarrow}
	\cO_{\bX^{1}\times \PP^{1}}(2).
\]
This map factors through
\begin{equation}\label{eqn:firstmodification}
	6\cO_{\bX^{1} \times \PP^{1}} \to
	\cO_{\bX^{1} \times \PP^{1}}(2)(-\bY_{0}^{1}),
\end{equation}
since $F_{0}|_{\bY_{0} \times \PP^{1}}$ is a zero map and so $(\pi_{1} \times \mathrm{id})^{*}F_{0}|_{\bY_{0}^{1} \times \PP^{1}}$ is zero, too. Let $F_{1}$ be the map in \eqref{eqn:firstmodification}. Then on $\bY_{0}^{1} \times \PP^{1} \setminus \overline{\bY}_{1}^{1} \times \PP^{1}$, $F_{1}$ is regular. Thus we have an extension $f_{1}$ of the map $f_{0} \circ (\pi_{1} \times \mathrm{id})$ such that the diagram
\[
	\xymatrix{\bX^{1} \times \PP^{1} \ar@{-->}[rrd]^{f_{1}}
	\ar[d]^{\pi_{1} \times \mathrm{id}}\\
	\bX^{0} \times \PP^{1} \ar@{-->}[r]^{f_{0}}
	& Gr(2, 4) \ar[r] & \PP^{5}}
\]
commutes.

Let $\pi_{2} \times \mathrm{id} : \bX^{2} \times \PP^{1} \to \bX^{1} \times \PP^{1}$ be the second blow up map. Then the base locus $\bB$ of $(\pi_{2} \times \mathrm{id})^{*}f_{1}$ is a two-to-one \'etale cover of $\overline{\bY}_{1}^{2}$, since on each fiber $\{M\} \times \PP^{1}$ for $M \in \overline{\bY}_{1}^{2}$, the undefined locus of $f_{1}|_{\{M\} \times \PP^{1}}$ is the union of two distinct points. In particular, it is a codimension two smooth subvariety. Let $\sigma : \Gamma \to \bX^{2} \times \PP^{2}$ be the blow up along $\bB_{1}$ and let $\bB_{2}$ be the exceptional divisor. Then $s : \Gamma \to \bX^{2} \times \PP^{2} \to \bX^{2}$ is a flat family of (possibly nodal) rational curves. Furthermore, the pull-back morphism
\[
	6\cO_{\Gamma}
	\stackrel{\pi^{*}(\pi_{2}\times \mathrm{id})^{*}F_{1}}{\to}
	\sigma^{*}\cO_{\bX^{2} \times \PP^{1}}(2)(-\bY_{0}^{2})
\]
factors through
\begin{equation}\label{eqn:secondmodification}
	6\cO_{\Gamma} \to
	\sigma^{*}\cO_{\bX^{2} \times \PP^{1}}(2)(-\bY_{0}^{2}-\bB_{2}).
\end{equation}
Let $F_{2}$ be the map in \eqref{eqn:secondmodification}. Then it is an epimorphism and we obtain a regular map $f_{2} : \Gamma \to \PP^{5}$ as below:
\[
	\xymatrix{\Gamma \ar[d]^{\sigma} \ar[rrddd]^{f_{2}}\\
	\bX^{2} \times \PP^{1} \ar[d]^{\pi_{2} \times \mathrm{id}}\\
	\bX^{1} \times \PP^{1} \ar@{-->}[rrd]^{f_{1}}
	\ar[d]^{\pi_{1} \times \mathrm{id}}\\
	\bX^{0} \times \PP^{1} \ar@{-->}[r]^{f_{0}}
	& Gr(2, 4) \ar[r] & \PP^{5}}
\]

Thus $(s : \Gamma \to \bX^{2}, f_{2} : \Gamma \to \PP^{5})$ defines a flat family of maps of degree 2 over $\bX^{2}$. By applying the standard stabilization of maps using relative log MMP, we obtain a family of stable maps $(\overline{s} : \overline{\Gamma} \to \bX^{2}, \overline{f}_{2} : \overline{\Gamma} \to \PP^{5})$. Moreover, $\overline{f}_{2}$ factors through $Gr(2, 4)$, since on an open dense subset $\bX^{s}$, the image is in a closed subvariety $Gr(2, 4)$. Hence we obtain a morphism $\Phi^{2} : \bX^{2} \to \bK$.

Finally, it is straightforward to see that this map is $G$-invariant, since $G$ acts as a change of coordinates of the domain curve. Thus we have a quotient map
\[
	\overline{\Phi}^{2} : \bX^{2}/G \to \bK.
\]
This is a birational morphism between two normal projective varieties. Moreover, both $\bX^{2}/G$ and $\bK$ are $\QQ$-factorial since they have only finite quotient singularities. They have the same Picard numbers, so $\overline{\Phi}^{2}$ is an isomorphism.
\end{proof}

Although the singularity is very mild, the space $\bK$ is still singular. Let $\bD \subset \bK$ be the locus of stable maps which are two-to-one maps to their images. This is precisely the locus of stable maps with nontrivial stabilizer groups. Along this locus, $\bK$ has $\ZZ_{2}$-quotient singularities.

In \cite[Theorem 1.2]{CHK12}, it is shown that the blow up of $\bK$ along $\bD$ is a smooth projective variety $\mathbf{CC}_{2}$, the so-called \emph{space of complete conics}.

\begin{coro}
The fiber product $\widetilde{\bM}_{2} := \bM_{2} \times_{\bR}\mathbf{CC}_{2}$ is a desingularization of $\bM_{2}$.
\end{coro}

\begin{proof}
The blow up centers for $\bM_{2} \to \bR$ and $\mathbf{CC}_{2} \to \bK \to \bR$ are disjoint. Therefore $\widetilde{\bM}_{2}$ is a blow up of $\mathbf{CC}_{2}$ at two smooth points.
\[
	\xymatrix{\widetilde{\bM}_{2} \ar[dd] \ar[r] &
	\mathbf{CC}_{2}\ar[d]\\
	& \bK \ar[d] \\ \bM_{2} \ar[r] & \bR.}
\]
\end{proof}

It would be very interesting to find a moduli-theoretic meaning of $\widetilde{\bM}_{2}$.

\begin{rema}
The log MMP of $\bK$ was studied in \cite{CC10}. The authors computed the stable base locus decomposition of the effective cone of $\bK$, and computed some of the modular birational models of $\bK$. In \cite{CC10}, some birational models were simply described as contractions of some divisors and their moduli-theoretic interpretations were not described. Two birational models $\bR$ and $\bX^{1}/G$, the first step of the partial desingularization, are indeed their ``modular'' interpretations. Indeed, $\bR$ is the image of $\phi_{H_{\sigma_{2}}}$ in \cite[Proposition 3.7]{CC10}. For $D \in c(H_{\sigma_{2}}T)$, the image of $\phi_{D}$ is $\bX^{1}/G$ (see \cite[Theorem 3.8]{CC10} for the notation).
\end{rema}


\bibliographystyle{amsplain}

\end{document}